\documentclass[10pt, leqno]{article}

\usepackage{amsmath,amssymb,amsthm, epsfig}

\usepackage{hyperref}

\usepackage{amsmath}

\usepackage[pdftex]{color}

\usepackage{color}

\usepackage{ulem}

\usepackage{mathrsfs}

\usepackage{wasysym}

\usepackage{stackrel}

\title{\large{\bf  Global weighted Lorentz estimates of oblique tangential derivative problems for weakly convex fully nonlinear operators}}
\author{\it by \smallskip \\
 Junior da S. Bessa\footnote{\noindent Universidade Federal do Cear\'{a}. Department of Mathematics. Fortaleza - CE, Brazil. \noindent \texttt{E-mail address: \url{junior.bessa@alu.ufc.br}}}, Gleydson C. Ricarte \footnote{\noindent Universidade Federal do Cear\'{a}. Department of Mathematics. Fortaleza - CE, Brazil. \noindent \texttt{E-mail address: \url{ricarte@mat.ufc.br}}}}

\newlength{\hchng}
\newlength{\vchng}
\def \R {\mathbb{R}}

\def \div {\mathrm{div}}

\def \loc {\mathrm{loc}}

\newcommand{\pe}{Question}
\newcommand{\defeq}{\mathrel{\mathop:}=}

\newtheorem{theorem}{Theorem}[section]

\newtheorem{lemma}[theorem]{Lemma}

\newtheorem{proposition}[theorem]{Proposition}

\newtheorem{corollary}[theorem]{Corollary}

\theoremstyle{definition}

\newtheorem{definition}[theorem]{Definition}

\newtheorem{example}[theorem]{Example}

\theoremstyle{remark}

\newtheorem{remark}[theorem]{Remark}

\numberwithin{equation}{section}

\newcommand{\intav}[1]{\mathchoice {\mathop{\vrule width 6pt height 3 pt depth  -2.5pt
\kern -8pt \intop}\nolimits_{\kern -6pt#1}} {\mathop{\vrule width
5pt height 3  pt depth -2.6pt \kern -6pt \intop}\nolimits_{#1}}
{\mathop{\vrule width 5pt height 3 pt depth -2.6pt \kern -6pt
\intop}\nolimits_{#1}} {\mathop{\vrule width 5pt height 3 pt depth
-2.6pt \kern -6pt \intop}\nolimits_{#1}}}

\begin{document}
\maketitle

\begin{abstract}

In this work, we develop weighted Lorentz-Sobolev estimates for viscosity solutions of fully nonlinear elliptic equations with oblique boundary condition under weakened convexity conditions in the following configuration:
$$
\left\{
\begin{array}{rclcl}
 F(D^2u,Du,u,x) &=& f(x)& \mbox{in} &   \Omega \\
 \beta \cdot Du + \gamma u&=& g &\mbox{on}& \partial \Omega,
\end{array}
\right.
$$
where $\Omega$ is a bounded domain in $\mathbb{R}^{n}$ ($n\ge 2$), under suitable assumptions on the source term  $f$, data $\beta, \gamma$ and $g$. In addition, we obtain Lorentz-Sobolev estimates for solutions to the obstacle problem and others applications.

\vspace{0.2cm}

\end{abstract}

\section{Introduction}

\hspace{0.4cm} 

The paper deals with global estimates of the tangential oblique derivative problem for fully nonlinear elliptic operators in the context of weighted Lorentz spaces (See Definition \ref{lorentzspaces}) under weak convexity conditions. More precisely, let $\Omega\subset \mathbb{R}^{n}$, $n\geq 2$, be a bounded domain with boundary regular and consider the following oblique boundary problem for second-order fully nonlinear elliptic equation:
\begin{equation}\label{1.1}
\left\{
\begin{array}{rclcl}
F(D^2u,Du,u,x) &=& f(x)& \mbox{in} &   \Omega \\
\mathfrak{B}(Du,u,x) &=& g(x) &\mbox{on}& \partial \Omega,
\end{array}
\right.
\end{equation}
where $F:\text{Sym(n)}\times \R^n \times \R \times \Omega \to \R$, $\text{Sym(n)}$ will denote the space of $n \times n$ real symmetric matrices, is assumed to be a second-order operator with uniformly elliptic structure, i.e., there exists {\it ellipticity} constants $0 < \lambda \le \Lambda < \infty$ such that 
\begin{equation}\label{Unif.Ellip.}
	\lambda\|\mathrm{Y}\|\leq F(\mathrm{X}+\mathrm{Y}, \varsigma, s, x)-F(\mathrm{X}, \varsigma, s, x) \leq \Lambda \|\mathrm{Y}\|
\end{equation}
for every $\mathrm{X}, \mathrm{Y} \in \text{Sym(n)}$ with $\mathrm{Y} \ge 0$ (it is understood that $Y$ is a non-negative matrix defined) and $(\varsigma, s, x)\in \mathbb{R}^n \times \R \times \Omega $, and $\mathfrak{B}(Du,u,x) = \beta(x) \cdot Du + \gamma(x) u(x)$ is a boundary operator, where $\beta:\partial \Omega\longrightarrow\mathbb{R}^{n}$ is a unit vector field and $\gamma, g:\partial \Omega\longrightarrow \mathbb{R}$ are a given real-functions. Under regularity suitable conditions on $\beta$,$\gamma$, $g$, we prove weighted Lorentz global estimates for Hessian of viscosity solutions of the problem \eqref{1.1}, when the source term belongs to the respective functional space, under asymptotic conditions of the operator $F$.

The boundary operator $\mathcal{B}$ is  describes by directional derivative in the direction of unit vector field $\beta$ defined in $\partial\Omega$ and the complementar function $\gamma$. Throughout this work, we will assume that there exists positive constant $\mu_{0}$ such that $\beta\cdot \textbf{n}\geq \mu_{0}$ in $\partial\Omega$, where is the outward normal vector of $\Omega$. Such a condition is necessary to guarantee that the problem \eqref{1.1} is well-posed due to the Shapiro–Lopatinskii complementary condition (cf. \cite{MaPaVi} for more details).

The interest in the problem with oblique condition is due to applications in areas beyond mathematics itself, such as Mathematical Physics and Applied Mathematics in general. For example in the theory of stochastic processes. Precisely, models of the form \eqref{1.1}  describes analitically in the interior of $\Omega$, by elliptic operator $F$, a strong Markov process with continuous paths such as Brownian motion. While, in the boundary operator $\mathfrak{B}$,    
\begin{eqnarray}\label{condicaoobliqua}
\mathfrak{B}(Du,u,x)=\beta(x)\cdot Du(x)+\gamma u(x)= g
\end{eqnarray}
the first term on the left side of \eqref{condicaoobliqua} describes the reflection of the process along of the points $\beta(x)$ which are transversal to $\partial\Omega$ and the diffusion along the boundary where $\beta$ becomes tangential to $\partial\Omega$. The second term is related to the absorption phenomenon of the process on the boundary. Furthermore, the theory of PDEs with oblique boundary conditions has a range of classic examples besides this one, such as other important applications in the mechanics of celestial bodies, stochastic control theory, capillary surfaces, reflected shocks in transonic flows and so on (cf. \cite{Lieberman}).
\paragraph{1.1. State-of-the-Art: nonlinear models and weighted Lorentz spaces.} Regarding nonlinear models and weighted Lorentz estimates, we will briefly comment on some relevant literature and their connections with interior/boundary regularity and general boundary conditions.

It is important to emphasize that convexity assumptions with respect to second-order derivatives play an essential role in obtaining of Hessian estimates. In the context of $W^{2,p}$-regularity to Fully nonlinear equations for the non-obstacle type problem, the complete study, with all of its variants, is a major line of investigation. The typical problem is
\begin{equation} \label{F1}
F(D^2 u)=0 \quad \textrm{in} \quad \mathrm{B}_1,
\end{equation}
where $F$ is uniformly elliptic, certainly is another very challenging and important subject of research in this new line of investigation. As a matter of fact, under a concavity (or convexity) assumption the seminal paper of Caffarelli, L.A. \cite{Caff1}.  Thus, a relevant problem in the theory is  to determine some structural conditions on $F$ other than the uniform ellipticity (in between concavity or convexity and no hypotheses) that guarantee higher order $W^{2,p}$. With this question in mind, a reasonable question in this context is the following: 
 \begin{center} 
 $\pe$:{\sl \, ``Which assumptions on $F$ guarantee that solutions of  \eqref{F1}, for $f$ is a suitable function space, $D^2 u \in L^p$?  }
 \end{center}
The general question of the best assumptions on $F$ ensuring this of regularity seemed unclear. The first advances in the question of find the minimal assumptions between convexity/concavity of the operator $F$ in the second derivatives was recently developed by Pimentel and Teixeira  in \cite{PT} in the context of $W^{2,p}$ interior estimates to \eqref{F1} and then da Silva and Ricarte in \cite{daSR19} in the global context (i.e., with boundary condition $u= \varphi$ in $\partial \Omega$).

In context of oblique boundary problem for fully nonlinear elliptic equations we can mention Byun and Han, in \cite{BJ} for $W^{2,p}$-global estimates for models of the form \eqref{1.1} with $\gamma=g=0$, under conditions of convexity in the governing operator. Inspired by this work, very recently, and mostly using the machinery developed in \cite{BJ}, Bessa et. al. obtained in \cite{Bessa},  $W^{2,p}$ regularity for $L^{p}$-viscosity solutions of \eqref{1.1} on a weaker convexity assumption for the second-order operator $F$.  Moreover, the authors also guaranteed some applications such as $p$-BMO type estimates for the Hessian, obstacle problem and density of solution (cf. \cite{BJ,BJ1} for
related results).

On the other hand, weighted Lorentz estimates for solutions of PDE's is a topic that has been the subject of interest for several researchers in recent years. This is due to the generalization of spaces in relation to Lebesgue spaces $L^{p}$'s and because they are functional spaces of transition between Lebesgue spaces themselves. This transitivity helps to find optimal regularity in some problems such as Lipschitz regularity, studied by Cianchi and Ma'zia, in \cite{CM}, for a class of  quasilinear elliptic equations  and  Daskalopoulos et al., in \cite{DKM}, for the fully nonlinear eliiptic equations. In the variational context there are several results focused on the theory of elliptic regularity. We highlight, Zhang et al.,in \cite{ZYY}, studied weak solutions of the following Dirichlet problem 
\begin{eqnarray}\label{reifenberg}
\left\{
\begin{array}{rclcl}
\div(ADu) &=& \div \vec{F}& \mbox{in} &   \Omega \\
u&=& 0 &\mbox{on}& \partial \Omega,
\end{array}
\right.
\end{eqnarray}
and guaranteed weighted variable Lorentz estimates for the gradient of solutions of the problem \eqref{reifenberg} solutions for Reifenberg domains and degenerate coefficients in weighted BMO spaces. In this area of problems, we still have other works such \cite{TZK,ZZ17,TN}.

In the non-variational context, we can mention \cite[Zhang and Zheng]{ZhangZhengmorrey}, guaranteed  estimates for solutions to the following Dirichlet problem
$$
\left\{
\begin{array}{rclcl}
F(D^2u,x) &=& f(x)& \mbox{in} &   \Omega \\
u&=& 0 &\mbox{on}& \partial \Omega,
\end{array}
\right.
$$   
for convex operator $F$ and source term $f$ belongs to weighted Lorentz and Lorentz-Morrey spaces. Recently in \cite[Zhang and Zheng]{ZhangZhenglorentz}, obtained, weighted Lorentz estimates for the viscosity solutions for the elliptic models of the form
$$
\left\{
\begin{array}{rclcl}
F(D^2u,Du,u,x) &=& f(x)& \mbox{in} &   \Omega \\
\beta \cdot Du&=& 0 &\mbox{on}& \partial \Omega,
\end{array}
\right.
$$   
under assumption of convexity/concavity of the governing operator. In this approach, an application of regularity was obtained in the context of variable exponent Morrey spaces.

Despite the extensive literature on fully nonlinear models with Dirichlet and Neumann boundary conditions, regularity properties in the functional spaces more general that $L^{p}$'spaces for models with general boundary conditions like \eqref{1.1} are much less investigated (cf. \cite{BessaOrlicz} \cite{ZhangZhenglorentz} and \cite{ZhangZhengmorrey} as examples in context of regularity theory). With this motivation, in this article, we provide a response to the \textit{Question} above with new regularity results for the problem (1.1) from the perspective of weighted Lorentz spaces.

We will summarize some recents developments of the weighted Lorentz regularity for the some problems (all with concave/convex structure):

\begin{table}[h]
\centering
\begin{tabular}{|c|c|c|}
\hline
{\it } & {\it Model equation}  & {\it Reference} \\
\hline
p-Dirichlet problem & 
$\left\{
	\begin{array}{rclcl}
		\div(ADu) &=& \div \vec{F}& \mbox{in} &   \Omega \\
		u&=& 0 &\mbox{on}& \partial \Omega,
	\end{array}
	\right.$ & \cite{ZYY} \\
\hline
Dirichlet problem & $
\left\{
\begin{array}{rclcl}
F(D^2u,x) &=& f(x)& \mbox{in} &   \Omega \\
u&=& 0 &\mbox{on}& \partial \Omega,
\end{array}
\right.
$    & \cite{ZhangZhengmorrey} \\
\hline
Oblique problem & $
\left\{
\begin{array}{rclcl}
	F(D^2u,Du,u,x) &=& f(x)& \mbox{in} &   \Omega \\
	\beta \cdot Du&=& 0 &\mbox{on}& \partial \Omega,
\end{array}
\right.
$    & \cite{ZhangZhenglorentz} \\
\hline
\end{tabular}\caption{Some recent developments of weighted Lorentz estimates}
\label{Table01}
\end{table}

\paragraph{1.2. Recession operator.}As reported before,  we will be interested in taking the limit porfile of elliptic operators. To do this, we define the following notation. 
\begin{definition} \label{DefR}
	For an operator $F$  we define the associated recession operator and use the notation $F^{\sharp}$ putting
	\begin{equation}\label{Reces}
		\displaystyle F^{\sharp}(\mathrm{X}, \varsigma, s, x) \defeq  \lim_{\tau \to 0^{+}} F_{\tau}(\mathrm{X}, \varsigma, s, x), \ (\mathrm{X}, \varsigma, s, x)\in \text{Sym(n)}\times \mathbb{R}^{n}\times \mathbb{R}\times \Omega,
	\end{equation}
where $F_{\tau}(\mathrm{X}, \varsigma, s, x)=\tau \cdot F\left(\tau^{-1}\mathrm{X}, \varsigma, s, x\right)$.
\end{definition}

A motivation  for \eqref{Reces} comes, for instance, from the analysis of  singular perturbation problems with singular operators on the left hand side. A simplified version of this class of problems is
\begin{equation}\label{E.L.}
	\left\{
	\begin{array}{rclcl}
		F(D^2u_{\epsilon}) &=& \zeta_{\epsilon}(u_{\epsilon})& \mbox{in} &   \Omega \\
		u_{\epsilon}&=& g(x) &\mbox{on}& \partial \Omega,
	\end{array}
	\right.
\end{equation}
where $F$ is a fully nonlinear elliptic operator and $\zeta_{\epsilon}$ is a suitable approximation to the Dirac-$\delta$ measure supported at the origin, have been recently considered in \cite{RT}.  An important accomplishment towards the free boundary problems are understanding of general free boundary conditions. Recently, Ricarte G.C. and Teixeira, E. V. in \cite{RT},   classified free boundary condition for the singular equation \eqref{E.L.} for rotational invariant fully nonlinear operators, $F$, i.e., for operators that depend only upon the eigenvalues of $D^2 u$ (these are the operators that commonly appear in differential geometry). The following free boundary condition hold in the viscosity sense: if $X_0 \in \partial_{red}\{u>0\}$ and ${\bf n}$ is the theoretical outward normal, then there exists an elliptic operator $F^{\sharp}$ (see Definition \ref{DefR}) for which the following free boundary condition holds:
$$
F^{\sharp}(X_0, Du(X_0) \otimes Du(X_0)) = 2 \int \zeta.
$$ 
Once this is verified, we can employ geometric measurement results to establish the smoothness $C^{1,\alpha}$ of the free boundary, closing the complete picture of the problem. 

It is important to emphasize that the Recession Operator means that the operator is almost an Pucci Operator near the infinity. The following example, illustrates the extension of such operators in the sense of good properties compared to the associated operator: 

\begin{example}
We consider be a perturbation of Bellman operators of the form
\begin{eqnarray*}\label{pertbellman}
\displaystyle F(\mathrm{X}, \varsigma, s, x) \defeq  \inf_{\iota \in \hat{\mathcal{A}}} \left( - \sum_{i,j=1}^{n} a^{\iota}_{ij}(x) \mathrm{X}_{ij} +  \sum_{i=1}^{n} b^{\iota}_i(x).\varsigma_i  + c^{\iota}(x)s\right) + \sum_{i=1}^{n} \textrm{arctg}(1+\lambda_i(\mathrm{X})).
\end{eqnarray*}
Here $(\lambda_i(\mathrm{X}))^{n}_{i=1}$ are eigenvalues of the matrix $\mathrm{X} \in \text{Sym(n)}$, $b_i^{\iota}, c^{\iota}: \Omega \to \R$ are real functions for each $\iota \in \mathfrak{I}$ (for $\mathfrak{I}$ be a set of index), and $\left( a^{\iota}_{ij}(x)\right)^{n}_{i,j=1} $ have eigenvalues in $[\lambda, \Lambda]$ for each $x \in \Omega$ and $\iota \in \mathfrak{I}$. In this case, $F$ is not an operator with a concave/convex structure, however, it is not difficult to check that
\begin{eqnarray*}
\displaystyle F^{\sharp}(\mathrm{X}, \varsigma, s, x)  & = &  \displaystyle \inf_{\iota \in \mathfrak{I}} \left( - \sum_{i,j=1}^{n} a^{\iota}_{ij}(x) \mathrm{X}_{ij}\right)
\end{eqnarray*}
which is a convex operator.
\end{example}

\begin{remark}
For more details and examples of operators and the recession operator, we recommend readers the works  \cite{Bessa}, \cite{daSR19},\cite{PT}, and \cite{RT}.
\end{remark}
In conclusion, the main objective of this work is to establish sharp Hessian weighted Lorentz estimates for problems that enjoy an asymptotically elliptic regular property like \eqref{Reces} in the governing operator in \eqref{1.1}. Finally, we will present some applications of this result in variable exponent Morrey spaces, density of weighted Lorentz-Sobolev spaces in fundamental classe (See Section \ref{section2} for the definitions) and a class of free boundary problems, the obstacle problem. In this case, we will obtain, within certain conditions, weighted Lorentz-Sobolev estimates for the obstacle problem with oblique boundary conditions.



It is worth mentioning that problems in the same line of research of the present work have been gaining prominence see \cite{BessaOrlicz}, \cite{Bessa},\cite{BJ0}, \cite{MF}, \cite{JP}, \cite{CS}, \cite{CS1} for an incomplete list of contributions.
\paragraph{1.3. Assumptions and main results.}We begin this part, by summarizing some classical notations that will be used throughout the paper:
\begin{itemize}
\item[\checkmark] For any $x=(x_{1},\ldots,x_{n-1},x_{n})\in\mathbb{R}^{n}$, we denote $x=(x',x_{n})$, where $x'=(x_{1},\ldots,x_{n-1})$;
\item [\checkmark] $\mathrm{B}_r(x)$ is the ball of radius $r>0$ centered at $x\in \mathbb{R}^{n}$. In particular, we write $\mathrm{B}_{r}=\mathrm{B}_{r}(0)$ and $\mathrm{B}^{+}_{r}=\mathrm{B}_{r}\cap\mathbb{R}^{n}_{+}$ 
\item[\checkmark] $\mathrm{B}_{r}^{+}(x)=\mathrm{B}^{+}_{r}+x_{0}$ is the half-ball of radius $r$ centered at $x$;
\item[\checkmark] we write $\mathrm{T}_r \defeq \{(x^{\prime},0) \in \mathbb{R}^{n-1} : |x^{\prime}| < r\}$ and $\mathrm{T}_r(x_0) \defeq \mathrm{T}_r + x^{\prime}_0$ where $x^{\prime}_0 \in \mathbb{R}^{n-1}$.  
\end{itemize}

We recall the definition of \textit{weights}. A function $\omega\in L^{1}_{loc}(\mathbb{R}^{n})$ is a \textit{weight} if it is non-negative and takes values in $(0,\infty)$ almost everywhere. This case, we identify $\omega$ with the measure 
\begin{eqnarray*}
\omega(E)=\int_{E}\omega(x)dx,
\end{eqnarray*} 
for Lebesgue measurable set $E\subset \mathbb{R}^{n}$. We say that be a weight $\omega$ belong to \textit{Muckenhoupt class} $\mathfrak{A}_{q}$ for $q\in(1,\infty)$, denoted by $\omega\in \mathfrak{A}_{q}$, if 
\begin{eqnarray*}
[\omega]_{q}\defeq \sup_{B\subset \mathbb{R}^{n}}\left(\intav{B}\omega(x)dx\right)\left(\intav{B}\omega(x)^{\frac{-1}{q-1}}dx\right)^{q-1}<\infty,
\end{eqnarray*}
where the supremum is taken over all balls $B\subset \mathbb{R}^{n}$.

Weights and $\mathfrak{A}_{q}$ classes are topics of importance in modern harmonic analysis and its applications. With these definition, we can define the main class of functions that we will work with throughout this article. 

\begin{definition}\label{lorentzspaces}
The \textit{weighted Lorentz space} $L^{p,q}_{\omega}(E)$ for $(p,q)\in (0,\infty)\times (0,\infty]$, Lebesgue measurable set $E\subset \mathbb{R}^{n}$ and weight $\omega$ is the set of all measurable functions $h$ in $E$ such that
\begin{eqnarray*}
\|h\|_{L^{p,q}_{\omega}(E)}\defeq \left(q\int_{0}^{\infty}t^{q-1}\omega(\{x\in E;|h(x)|>t\})^{\frac{q}{p}}dt\right)^{\frac{1}{q}}<\infty,
\end{eqnarray*} 
when $q\in(0,\infty)$, and 
\begin{eqnarray*}
\|h\|_{L^{p,\infty}_{\omega}(E)}\defeq\sup_{t>0}t\omega(\{x\in E;|h(x)|>t\})^{\frac{1}{p}}<\infty.
\end{eqnarray*}
Furthermore, the \textit{weighted Lorentz-Sobolev} space $W^{k}L^{p,q}_{\omega}(E)$ (for $k\in\mathbb{N}$) is the set of all measurable functions $h$ in $E$ such that all distributional derivates $D^{\alpha}h$, for multiindex $\alpha$ with length $|\alpha|=0,1,\ldots,k$ also belong to $L^{p,q}_{\omega}(E)$ with norm is given by
\begin{eqnarray*}
	\|h\|_{W^{k}L^{p,q}_{\omega}(E)}\defeq \sum_{|\alpha|\leq k}\|D^{\alpha}h\|_{L^{p,q}_{\omega}(E)}.
\end{eqnarray*} 
\end{definition}
\begin{remark}
Under the Definition \ref{lorentzspaces}:
\begin{itemize}
\item [\checkmark] In case that $\omega\equiv1$, $L^{p,q}_{\omega}(E)$ is the Lorentz spaces;
\item[\checkmark] If $q=\infty$ and $\omega\equiv 1$ then $L^{p,\infty}_{\omega}(E)$ is called by weak $L^{p}$ space and denoted by $L^{p}_{w}(E)$.
\item[\checkmark] If $p=q$ and $\omega\equiv1$ the weighted Lorentz space $L^{p,q}_{\omega}(E)$ coincides with classical Lebesgue space $L^{p}(E)$ by Layer-Cake Representation. 
\end{itemize} 
\end{remark}

Throughout this manuscript, we will assume the following structural conditions on the problem $\eqref{1.1}$:
\begin{enumerate}
\item[(E1)] ({\bf Operator Structure}) We assume that $F \in C^0(\text{Sym(n)}, \R^n, \R, \Omega)$. Moreover, there are constants $0 < \lambda \le \Lambda$, $\sigma\geq 0$ and $\xi\geq 0$ such that
\begin{eqnarray}\label{EqUnEll}
\mathcal{M}^{-}_{\lambda,\Lambda}(\mathrm{X}-\mathrm{Y}) - \sigma |\zeta_1-\zeta_2| -\xi|r_1-r_2| &\le& F(\mathrm{X}, \zeta_1,r_1,x)-F(\mathrm{Y},\zeta_2,r_2,x) \nonumber \\
&\le& \mathcal{M}^{+}_{\lambda, \Lambda}(\mathrm{X}-\mathrm{Y})+ \sigma |\zeta_1-\zeta_2| + \xi|r_1-r_2| \label{5}
\end{eqnarray}
for all $\mathrm{X},\mathrm{Y} \in \text{Sym(n)}$, $\zeta_1,\zeta_2 \in \mathbb{R}^n$, $r_1,r_2 \in \mathbb{R}$, $x \in \Omega$, where
\begin{equation}
\mathcal{M}^{+}_{\lambda,\Lambda}(\mathrm{X}) \defeq  \Lambda \sum_{e_i >0} e_i +\lambda \sum_{e_i <0} e_i \quad \text{and} \quad \mathcal{M}^{-}_{\lambda,\Lambda}(\mathrm{X}) \defeq \Lambda \sum_{e_i <0} e_i + \lambda \sum_{e_i >  0} e_i,
 \end{equation}
 are the \textit{Pucci's extremal operators} and $e_i = e_i(\mathrm{X})$ ($1\leq i\leq n$) denote the eigenvalues of $\mathrm{X}$. By normalization question we assume that $F(0,0,0,x)=0$ for all $x\in \Omega$.

\item[(E2)] ({\bf Regularity of the data}) The data satisfy $f \in L^{p,q}_{\omega}(\Omega)$ for $ (p,q)\in (n,\infty)\times(0,+\infty]$, $g, \gamma \in C^{1,\alpha}(\partial \Omega)$ with $\gamma \le 0$ and $\beta \in C^{1,\alpha}( \partial \Omega; \mathbb{R}^n)$ (for some $\alpha\in (0,1)$).

\item[(E3)] ({\bf Oscilation  of the coefficents}) For a fixed $x_0\in \Omega$, recall of the quantity:
$$
\Psi_{F}(x; x_0) \defeq \sup_{\mathrm{X} \in \text{Sym(n)}} \frac{|F(\mathrm{X},0,0,x) - F(\mathrm{X},0,0,x_0)|}{1+\|\mathrm{X}\|},
$$
measures the oscillation of the coefficients of $F$ around $x_0$. When, $x_{0}=0$ we abreviate following notation  $\Psi_{F}(x, 0) = \Psi_F(x)$.  We assume that the oscilation function  $\Psi_{F^{\sharp}}$ is assumed to be H\"{o}lder continuous function in the $L^{p}$-average sense for every $\mathrm{X} \in \text{Sym(n)}$. This means that, there exist universal constants (here for universal constants, we means that if it depends only on $n, \lambda, \Lambda, p, \mu_0, \|\gamma\|_{C^{1,\alpha}(\partial \Omega)}$ and $\|\beta\|_{C^{1,\alpha}(\partial \Omega)}$) $\hat{\alpha} \in (0,1)$, $\theta_0 >0$ and $0 < r_0 \le 1$ such that
$$
\left( \intav{\mathrm{B}_r(x_0) \cap \Omega} \Psi_{F^{\sharp}}(x,x_0)^{p} dx \right)^{1/p} \le \theta_0 r^{\hat{\alpha}}
$$
for $x_0 \in \overline{\Omega}$ and $0 < r \le r_0$.

\item[(E4)] ({\bf Interior estimates}) We assume that the recession operator $F^{\sharp}$ there exists and has a priori $C^{1,1}_{loc}$ estimates. More precisely, if $\mathfrak{h}$ is viscosity solution of $F^{\sharp}(D^{2}\mathfrak{h}) = 0$ in $\mathrm{B}_{1}$, then $\mathfrak{h}\in C^{1,1}(\mathrm{B}_{\frac{1}{2}})$ and
\begin{eqnarray*}
\|\mathfrak{h}\|_{C^{1,1}(\mathrm{B}_{\frac{1}{2}})}\leq \mathfrak{c_{1}}\|\mathfrak{h}\|_{L^{\infty}(\mathrm{B}_{1})}
\end{eqnarray*}
for a constant $\mathfrak{c_{1}}>0$.

\item[(E5)] ({\bf Up-to-the boundary a priori estimates}) We shall assume that the recession operator $F^{\sharp}$ there exists and fulfils a up-to-the boundary $C^{1,1}$ \textit{a priori} estimates, i.e., for each $x_0 \in \mathrm{B}^{+}_{1}$ and  $g_0 \in C^{1,\alpha}(\mathrm{T}_1)$ (for some $\alpha \in (0, 1)$), there exists a solution $\mathfrak{h} \in C^{1,1}(\mathrm{B}^+_1) \cap C^0(\overline{\mathrm{B}^+_1})$ of
$$
\left\{
\begin{array}{rclcl}
 F^{\sharp}(D^2 \mathfrak{h},x_0) &=& 0& \mbox{in} &   \mathrm{B}^+_1 \\
 \mathfrak{B}(D\mathfrak{h},\mathfrak{h},x)&=& g_{0}(x) &\mbox{on}& \mathrm{T}_1
\end{array}
\right.
$$
such that
$$
\|\mathfrak{h}\|_{C^{1,1}\left(\overline{\mathrm{B}^{+}_{\frac{1}{2}}}\right)} \le \mathfrak{c_{2}} \left(\|\mathfrak{h}\|_{L^{\infty}(\mathrm{B}^{+}_1)}+\|g_0\|_{C^{1,\alpha}(\overline{\mathrm{T}_1})}\right)
$$
for some constant $\mathfrak{c_{2}}>0$.
\end{enumerate}

\begin{remark}
From now on, an operator fulfilling $\text{(E1)}$ will be called $(\lambda, \Lambda, \sigma, \xi)-$elliptic operator. Moreover, the conditions (E4) and (E5) are valid everywhere when $F$ is convex/concave operator. 
\end{remark} 

We now state our main result, which endures global $W^{2}L^{p,q}_{\omega}$ estimates for viscosity solutions of the problem \eqref{1.1}.
\vspace{0.4cm}

\begin{theorem}[{\bf $W^{2}L^{p,q}_{\omega}$ - global estimates}]\label{T1}
Let  $\Omega\subset \mathbb{R}^{n}$ ($n\geq 2$) be a bounded domain with $\partial \Omega \in C^{2,\alpha}$ (for some $\alpha\in(0,1)$) and $\omega\in \mathfrak{A}_{\frac{p}{n}}$ weight. Assume that structural assumptions (E1)-(E5) are in force and $u$ be an $L^{\tilde{p}}-$viscosity solution of \eqref{1.1} with $\tilde{p}\in[n,p)$. Then, $u \in W^{2}L^{p,q}_{\omega}(\Omega)$ with estimate \begin{equation} \label{2}
\|u\|_{W^{2}L^{p,q}_{\omega}(\Omega)} \le \mathrm{C}\cdot\left(\|f\|_{L^{p,q}_{\omega}(\Omega)}+\| g\|_{C^{1,\alpha}(\partial \Omega)}   \right),
\end{equation}
where $\mathrm{C}$ is a positive constant that depends only on $n$,  $\lambda$,  $\Lambda$, $\xi$, $\sigma$, $p$, $q$, $\tilde{p}$, $[\omega]_{\frac{p}{n}}$, $\mu_{0}$, $\mathfrak{c_{1}}$,  $\mathfrak{c_{2}}$, $\|\beta\|_{C^{1,\alpha}(\partial\Omega)}$, $\|\gamma\|_{C^{1,\alpha}(\partial\Omega)}$ and $\|\partial \Omega\|_{C^{1,1}}$.
\end{theorem}

Our approach is developed from the following perspective: Firstly, we will study the set where the Hessian is "bad" and we will prove that the measure (with respect to $\omega$) has a power decay of $t$. To do this, we will work in a neighborhood of the recession operator and use Tangential Analysis arguments to obtain, via the estimates assumed in ($E4)$ and $(E5)$, such decay. After this step, we guarantee that the Hessian exists and belongs to the Lorentz space with adequate weight and together with estimates for the gradient and the function itself, via the covering argument, we obtain the desired result. We emphasize that our proof is based on the method employed in \cite{Bessa}. However, our results extend regarding oblique boundary scenario, former results from  \cite{ZhangZhenglorentz},\cite{ZhangZhengmorrey}, \cite{BOW16}, \cite{PT}  \cite{ST} (see also \cite{daSR19}), and to some extent, those from \cite{BOW16}, \cite{Kry13}, \cite{Kry17} and \cite{ZZZ21} by making using of techniques adapted to the general framework of the fully nonlinear models under relaxed convexity assumptions and oblique boundary data.

The rest of this manuscript is organized as follows.: Section \ref{section2} contains the main notations and preliminary results where stands out a review of the theory of weighted Lorentz spaces and approximation method of the recession operator. In Section \ref{section3} we obtain local interior and up-to-boundary estimates  weighted Lorentz-Sobolev estimates and with these estimates we present the proof of the Main Theorem \ref{T1}. In the last section \ref{applications}, we present some applications of the results obtained in section \ref{section3}. More precisely, we obtain Morrey regularity with variable exponents for solutions of \eqref{1.1} in the scenario where $\gamma=g=0$. After, we turn our attention to the density of $W^{2}L^{p,q}_{\omega}$ in the fundamental class $S$ of solutions to the flat problem. Finally, we ensure, in context of the obstacle problem,  Lorentz-Sobolev estimates of solutions to this problem.

\section{Preliminaries and auxiliary results} \label{section2}

\hspace{0.4cm}We introduce the appropriate notions of viscosity solutions to \eqref{1.1}. For this, consider the problem 
\begin{equation}\label{2.1}
	\left\{
	\begin{array}{rclcl}
		F(D^2u,Du,u,x) &=& f(x)& \mbox{in} &   \Omega \\
		\mathfrak{B}(Du,u,x)&=& g &\mbox{on}& \Gamma,
	\end{array}
	\right.
\end{equation}
where $\Gamma \subset \Omega$ is a relatively open set.
\begin{definition}[{\bf $C^{0}-$viscosity solutions}]\label{VSC_0} Let $F$ be a $\left(\lambda, \Lambda, \sigma, \xi\right)-$elliptic operator and let $f\in C^{0}(\Omega)$. A function $u \in C^{0}(\Omega \cup \Gamma)$  is said a $C^{0}-$viscosity solution of \eqref{2.1} if the following condition hold:
\begin{enumerate}
\item[a)] for all $ \forall\,\, \varphi \in C^{2} (\Omega \cup \Gamma)$ touching $u$ by above  at  $x_0 \in \Omega \cup \Gamma$,
$$
\left\{
\begin{array}{rcl}
F\left(D^2 \varphi(x_{0}), D \varphi(x_{0}), \varphi(x_{0}), x_{0}\right)  \geq f(x_0) & \text{when} & x_0 \in \Omega \\
\mathfrak{B}(D\varphi(x_{0}),\varphi(x_{0}),x_{0}) \ge g(x_0) & \text{when} & x_0 \in \Gamma
\end{array}
\right.
$$
\item[b)] for all $ \forall\,\, \varphi \in C^{2 } (\Omega \cup \Gamma)$ touching $u$ by below  at  $x_0 \in \Omega \cup \Gamma$,
$$
\left\{
\begin{array}{rcl}
F\left(D^2 \varphi(x_{0}), D \varphi(x_{0}), \varphi(x_{0}), x_{0}\right)  \leq f(x_0) & \text{when} & x_0 \in \Omega \\
\mathfrak{B}(D\varphi(x_{0}),\varphi(x_{0}),x_{0}) \le g(x_0) & \text{when} & x_0 \in \Gamma
\end{array}
\right.
$$

\end{enumerate}
\end{definition}
\begin{remark}
If we use $W^{2,\tilde{p}}_{loc}$ for some $\tilde{p}\in \left(\frac{n}{2},\infty\right)$ functions instead of functions in $C^{2}_{loc}$ as test functions in Definition
\ref{VSC_0} and $f\in L^{\tilde{p}}(\Omega)$, we call such a defined solution a \textbf{$L^{\tilde{p}}$-viscosity solution}. Moreover, we point that such definitions are equivalent when $F$ and $f$ are continuous functions (cf. \cite{CCKS} for more details). 
\end{remark}

Now, we recall in the fundamental class of solutions of the uniformly elliptic equations (cf. \cite{CC} and \cite{CCKS}).

\begin{definition}
 We define the class $\overline{S}\left(\lambda,\Lambda,\sigma, f\right)$ and $\underline{S}\left(\lambda,\Lambda,\sigma, f\right)$ to be the set of all continuous functions $u$ that satisfy $L^{+}(u) \ge f$, respectively $L^{-}(u) \le f$ in the viscosity sense (see Definition \ref{VSC_0}), where
 $$
 L^{\pm}(u) \defeq \mathcal{M}^{\pm}_{\lambda,\Lambda}(D^2 u) \pm \sigma |Du|.
 $$
 Furthermore, we define
 $$
 S\left(\lambda, \Lambda, \sigma,f\right) \defeq  \overline{S}\left(\lambda, \Lambda, \sigma,f\right) \cap \underline{S}\left(\lambda, \Lambda,\sigma, f\right)\,\,\text{and}\,\,
 S^{\star}\left(\lambda, \Lambda,\sigma, f\right) \defeq  \overline{S}\left(\lambda, \Lambda, \sigma,|f|\right) \cap \underline{S}\left(\lambda, \Lambda,\sigma, -|f|\right).
  $$
  Moreover, when $\sigma=0$, we denote $S^{\star}(\lambda,\Lambda,0,f)$ just by $S^{\star}(\lambda,\Lambda,f)$ (resp. by $S, \underline{S}, \overline{S}$).
 \end{definition}

For our approach, we need the concept of paraboloids and the tangency of these objects  and continuous functions.
\begin{definition}
We recall that a \textit{paraboloid with openning} $M>0$ is the function $P_{M}(x)=a+b\cdot x\pm\frac{M}{2}|x|^{2}$. When the sign of $P_{M}$ is "+" see that $P_{M}$ is a convex function and  a concave function otherwise.
\end{definition}
 Let $u$ be a continuous function in $\Omega$. We define  for any  open set $\widetilde{\Omega}\subset \Omega$ and $M>0$ 
$$
\underline{G}_{\mathrm{M}}(u,\widetilde{\Omega}) \defeq \left\{x_0 \in \widetilde{\Omega} ; \exists \, \mathrm{P}_{\mathrm{M}} \,\,\, \textrm{concave parabolid s. t.} \,\,\, \mathrm{P}_{\mathrm{M}}(x_0)=u(x_0), \,\, \mathrm{P}_{\mathrm{M}}(x) \le u(x)\,\, \forall \, x \in \widetilde{\Omega}\right\}
$$
and
$$
\underline{\mathcal{A}}_{\mathrm{M}}(u,\widetilde{\Omega}) \defeq \widetilde{\Omega} \setminus \underline{G}_{\mathrm{M}}(u,\widetilde{\Omega}).
$$

Analogously we define , using concave paraboloids, $\overline{G}_{\mathrm{M}}(u,\widetilde{\Omega})$ and  $\overline{\mathcal{A}}_{\mathrm{M}}(u,\widetilde{\Omega})$. Also define the sets
$$
G_{\mathrm{M}}(u,\widetilde{\Omega}) \defeq  \underline{G}_{\mathrm{M}}(u,\widetilde{\Omega}) \cap \overline{G}_{\mathrm{M}}(u,\widetilde{\Omega})\,\,\,\text{and}\,\,\,\mathcal{A}_{\mathrm{M}}(u,\widetilde{\Omega}) \defeq \underline{\mathcal{A}}_{\mathrm{M}}(u,\widetilde{\Omega}) \cap \overline{\mathcal{A}}_{\mathrm{M}}(u,\widetilde{\Omega}).
$$
Furthermore, we define:
\begin{eqnarray*}
\overline{\Theta}(u,\widetilde{\Omega})(x) &\defeq& \inf\left\{\mathrm{M}>0 ; \, x \in \overline{G}_{\mathrm{M}}(u,\widetilde{\Omega})\right\},\\
\underline{\Theta}(u,\widetilde{\Omega})(x)&\defeq& \inf\left\{\mathrm{M}>0 ; \, x \in \underline{G}_{\mathrm{M}}(u,\widetilde{\Omega})\right\},\\
\Theta(u,\widetilde{\Omega})(x) &\defeq& \sup\left\{\underline{\Theta}(u,\widetilde{\Omega})(x), \overline{\Theta}(u,\widetilde{\Omega})(x)\right\}.
\end{eqnarray*}

\begin{remark}
We observe that $\Theta(u,\tilde{\Omega})$ is a mensurable function (see \cite{CC} for more details and facts on the function $\Theta$). 
\end{remark}
For compactness and stability arguments we will need the following theorem whose proof can be found in \cite[Theorem 3.8]{CCKS}:
\begin{lemma}[{\bf Stability Lemma}]\label{Est}
For $k \in \mathbb{N}$ let $\Omega_k \subset \Omega_{k+1}$ be an increasing sequence of domains and $\displaystyle \Omega \defeq \bigcup_{k=1}^{\infty} \Omega_k$. Let $p \geq n$ and $F, F_k$ be $(\lambda, \Lambda, \sigma, \xi)-$elliptic operators. Assume $f \in L^{p}(\Omega)$, $f_k \in L^p(\Omega_k)$ and that $u_k \in C^0(\Omega_k)$ are $L^{p}-$viscosity sub-solutions (resp. super-solutions) of
$$
	F_k(D^2 u_k,Du_k,u_k,x)=f_k(x) \quad \textrm{in} \quad \Omega_k.
$$
Suppose that $u_k \to u_{\infty}$ locally uniformly in $\Omega$ and that for $\mathrm{B}_r(x_0) \subset \Omega$ and $\varphi \in W^{2,p}(\mathrm{B}_r(x_0))$ we have
\begin{equation} \label{Est1}
	\|(\hat{g}-\hat{g}_k)^+\|_{L^p(\mathrm{B}_r(x_0))} \to 0 \quad \left(\textrm{resp.} \,\,\, \|(\hat{g}-\hat{g}_k)^-\|_{L^p(\mathrm{B}_r(x_0))} \to 0 \right),
\end{equation}
where $\hat{g}(x) \defeq F(D^2 \varphi, D \varphi, u,x)-f(x)$ and $\hat{g}_k(x) =  F_k(D^2 \varphi, D \varphi, u_{k},x)-f_k(x)$.  Then $u$ is an $L^{p}-$viscosity sub-solution (resp. super-solution) of
$$
	F(D^2 u,Du,u,x)=f(x) \quad \textrm{in} \quad \Omega.
$$
Moreover, if $F, f$ are continuous, then $u$ is a $C^0-$viscosity sub-solution (resp. super-solution) provided that \eqref{Est1} holds for all $\varphi \in C^2(\mathrm{B}_r(x_0))$ test function.
\end{lemma}

In this context, of compactness and stability arguments, the following ABP Maximum Principle will also be extremely important (see \cite[Theorem 2.1]{LiZhang} for a proof).
\begin{lemma}[{\bf ABP Maximum Principle}]\label{ABP-fullversion}
Let $\Omega\subset \mathrm{B}_{1}$ and $u$ satisfying
\begin{equation*}
\left\{
\begin{array}{rclcl}
u\in \mathcal{S}(\lambda,\Lambda,f) &\mbox{in}&   \Omega \\
\beta\cdot Du+\gamma u=g  &\mbox{on}&  \Gamma.
\end{array}
\right.
\end{equation*}
Suppose that exists $\varsigma\in \partial \mathrm{B}_{1}$ such that $\beta\cdot\varsigma\geq \mu_0$ and $\gamma\le 0$ in $\Gamma$. Then,
	\begin{eqnarray*}
		\|u\|_{L^{\infty}(\Omega)}\leq \|u\|_{L^{\infty}(\partial \Omega\setminus \Gamma)}+\mathrm{C}(n, \lambda, \Lambda, \mu_0)(\| g\|_{L^{\infty}(\Gamma)}+\| f\|_{L^{n}(\Omega)})
	\end{eqnarray*}
\end{lemma}

In the sequel, we comment on the existence and uniqueness of viscosity solutions with oblique boundary conditions. For that purpose, we consider the following problem:
\begin{eqnarray}\label{existence}
\left\{
\begin{array}{rclcl}
F(D^2u, x) &=& f(x) & \mbox{in} & \Omega,\\
\beta\cdot Du+\gamma u&=& g & \mbox{on} & \Gamma\\
u(x)&=&\varphi(x) & \mbox{in} & \partial \Omega\setminus \Gamma,
\end{array}
\right.
\end{eqnarray}
where $\Gamma$ is relatively open of $\partial\Omega$. In \cite{Bessa} can be seen a proof for the next theorem. For this reason, we will omit it here.

\begin{theorem}[{\bf Existence and Uniqueness}]\label{Existencia}
Let $\Gamma\in C^{2}$, $\beta\in C^{2}(\overline{\Gamma})$, $\gamma\le 0$ and $\varphi\in C^{0}(\partial \Omega\setminus \Gamma)$. Suppose that there exists $\varsigma\in\partial \mathrm{B}_{1}$ such that $\beta\cdot \varsigma\ge \mu_0$ on $\Gamma$ and assume that there exists be a modulus of continuity $\rho$ for the oscilation coefficients, that is, $\rho$ is nondecreasing with $ \displaystyle \lim_{t \to 0} \rho(t) =0$ and
$$
\Psi_{F}(x,y) \le \rho(|x-y|).
$$
In addition, suppose that $\Omega$ satisfies an exterior cone condition at any $x\in \partial \Omega\setminus \overline{\Gamma}$ and satisfies an exterior sphere condition at any $x\in \overline{\Gamma}\cap (\partial \Omega\setminus \Gamma)$. Then, there exists a unique viscosity solution $u\in C^{0}(\overline{\Omega})$ of \eqref{existence}.
\end{theorem}

\begin{remark}
For more results on the existence and uniqueness of solutions in the sense of viscosity of problems with oblique edge conditions, we recommend to readers the works by Hishi and Lions \cite{Hi} and Lierberman's Book \cite{Lieberman}.
\end{remark}
\subsection{A review on  weights and weighted Lorentz spaces}

\hspace{0.4cm} We list some properties already known about weights and weighted Lorentz spaces that will be indispensable for the course of this work. First, on weight we have the following lemma whose proof we found in chapter One of  \cite{Turesson}.

\begin{lemma}\label{propriedadesdospesos}
Let $\omega$ be an $\mathfrak{A}_{s}$ weight for some $s\in(1,\infty)$. Then,
\begin{enumerate}
\item[(a)](increasing) If $r\ge s$, then $\omega$ belong to class $\mathfrak{A}_{r}$ and $[\omega]_{r}\leq [\omega]_{s}$. 
\item[(b)] (open-end) There exists a small enough constant $\varepsilon_{0}>0$ depending  only on $n$, $s$ and $[\omega]_{s}$ such that $\omega\in \mathfrak{A}_{s-\varepsilon_{0}}$ with $s-\varepsilon_{0}>1$.
\item[(c)] (strong doubling) There exist two postive constants $c_{1}$ and $\theta\in(0,1)$ depending only on $n$, $s$ and $[\omega]_{s}$ such that 
\begin{eqnarray*}
\frac{1}{[\omega]_{s}}\left(\frac{|E|}{|\Omega|}\right)^{s}\leq \frac{\omega(E)}{\omega(\Omega)}\leq c_{1}\left(\frac{|E|}{|\Omega|}\right)^{\theta}.
\end{eqnarray*}  
for all $E\subset\Omega$ Lebesgue measurable set. 
\end{enumerate} 
\end{lemma}

On weighted Lorentz spaces, under certain conditions we can  continuously embedded on the Lebesgue spaces. This is the content of the following lemma, whose proof we found in \cite[Lemma 2.10]{ZhangZhenglorentz}
\begin{lemma}\label{Incluaolorentzlebesgue}
Let $(p,q)\in(n,\infty)\times (0,\infty]$, $\omega$ be an $\mathfrak{A}_{\frac{p}{n}}$ weight and bounded measurable set $E\subset \mathbb{R}^{n}$. If $f\in L^{p,q}_{\omega}(E)$, then for any $r\in[n,p)$, we have $f\in L^{r}(E)$ with the following estimate
\begin{eqnarray*}
\|f\|_{L^{r}(E)}\leq \mathrm{C}_{p,q} \|f\|_{L^{p,q}_{\omega}(E)},
\end{eqnarray*}
where $\mathrm{C}_{p,q}$ is a positive constant depending only on $n$, $p$, $r$, $[\omega]_{\frac{p}{n}}$ and $|E|$. 
\end{lemma}
 
 We recall of Hardy-Littlewood maximal function. For $f\in L^{1}_{\loc}(\mathbb{R}^{n})$, the Hardy-Littlewood maximal function  de $f$ is defined by
\begin{eqnarray*}
	\mathcal{M}(f)(x)=\sup_{r>0}\intav{B_{r}(x)}|f(y)|dy.
\end{eqnarray*}

Mengesha and Phuc proved an version of classical Hardy-Littlewood-Wiener theorem for Weight Lorentz spaces in \cite[Lemma 3.11]{Mengesha} that we will use later.
\begin{lemma}\label{Hardylpq}
	Let $\omega$ be an $\mathfrak{A}_{s}$ weight for $s\in(1,\infty)$. Then for any $q\in(0,\infty]$ there exists a positive constant C depending only on $n$, $s$, $q$, $[\omega]_{s}$ such that 
	\begin{eqnarray}\label{hardy}
		\|\mathcal{M}f\|_{L^{s,q}_{\omega}(\mathbb{R}^{n})}\leq C\|f\|_{L^{s,q}_{\omega}(\mathbb{R}^{n})}
	\end{eqnarray}
	for all $f\in L^{s,q}_{\omega}(\mathbb{R}^{n})$. Conversely, if inequality \eqref{hardy} holds for all $f\in L^{s,q}_{\omega}(\mathbb{R}^{n})$, then $\omega$ must  be an weight. 
\end{lemma} 

Later we will need a sufficient condition for the Hessian in the distributional sense to be in weighted Lorentz spaces. This is guaranteed from the following lemma (for proof cf. \cite[Lemma 2.2]{ZhangZhengmorrey}).

\begin{lemma}\label{caracterizationofhessian}
Let $(p,q)\in(1,\infty)\times (0,\infty]$ and $\omega\in A_{s}$ weight for some $s\in(1,\infty)$. Assume that $u\in C^{0}(E)$ for bounded domain $E\subset \mathbb{R}^{n}$ and set for $r>0$
\begin{eqnarray*}
\Theta(u,r)(x)\defeq  \Theta(u,B_{r}(x)\cap E)(x), \ x\in E.
\end{eqnarray*}
If $\Theta(u,r)\in L^{p,q}_{\omega}(E)$, then Hessian in the distributional sense $D^{2}u\in L^{p,q}_{\omega}(E)$ with estimate
\begin{eqnarray}
\|D^{2}u\|_{L^{p,q}_{\omega}(E)}\leq C(n)\|\Theta(u,r)\|_{L^{p,q}_{\omega}(E)}.
\end{eqnarray} 
\end{lemma}

Finally, we will need the following standard characterization of functions in weighted Lorentz spaces whose proof goes back to standard measure theory arguments, which we will present below as a courtesy to our readers.  (cf. \cite[Lemma 3.12]{Mengesha}).
\begin{proposition}\label{P1}
Let $\omega$ an $\mathfrak{A}_{s}$ weight for some $s\in(1,\infty)$, $h: E \to \R$ be a nonnegative measurable function in a bounded domain $E\subset \mathbb{R}^{n}$. Let $\eta>0$ and $\mathrm{M} >1$ constants. Then, 
$$
h \in L^{p,q}_{\omega}(E) \Longleftrightarrow  \sum_{j=1}^{\infty} \mathrm{M}^{qj} \omega(\{x\in E; h(x)>\eta M^{j}\})^{\frac{q}{p}}\defeq S< \infty
$$
for every $p,q\in(0,\infty)$ and 
\begin{eqnarray}\label{characterization1} 
\mathrm{C}^{-1}S\le \|h\|_{L^{p,q}_{\omega}(E)}^{q} \le \mathrm{C}(\omega(E)^{\frac{q}{p}}+S),
\end{eqnarray}
with $\mathrm{C}=\mathrm{C}(q,\eta, \mathrm{M})$ is positive constant. Moreover, for $p\in(0,\infty)$ and $q=\infty$ we have
\begin{eqnarray*}
\overline{C}^{-1}T\leq \|h\|_{L^{p,\infty}_{\omega}(E)}\leq \overline{C}(\omega(E)^{\frac{1}{p}}+T),
\end{eqnarray*}
where $\overline{C}=\overline{C}(\eta, \mathrm{M})$ is positive constant and 
\begin{eqnarray}\label{characterization2}
T=\sup_{j\in\mathbb{N}}M^{j}\omega(\{x\in E; h(x)>\eta M^{j}\})^{\frac{1}{p}}.
\end{eqnarray} 
\end{proposition}
\begin{proof}
By simplicity, we use the notation  $\omega_{h}(t)=\omega(\{x\in E; h(x)>t\})$ for $t>0$. For first part, assume that $h\in L^{p,q}_{\omega}(E)$. We observe that, by function $t\longmapsto \omega_{h}(t)^{\frac{q}{p}}$ be a non-increasing and $M>1$,
\begin{eqnarray*}
\|h\|_{L^{p,q}_{\omega}(E)}^{q}&=&q\int_{0}^{\eta }t^{q-1}\omega_{h}(t)^{\frac{q}{p}}dt+q\displaystyle\int_{\dot{\bigcup}_{k=0}^{\infty}(\eta M^{k},\eta M^{k+1}]}t^{q-1}\omega_{h}(t)^{\frac{q}{p}}dt\\
&\geq&\displaystyle\sum_{k=0}^{\infty}q\int_{\eta M^{k}}^{\eta M^{k+1}}t^{q-1}\omega_{h}(t)^{\frac{q}{p}}dt\geq\displaystyle\sum_{k=0}^{\infty}q\int_{\eta M^{k}}^{\eta M^{k+1}}t^{q-1}\omega_{h}(\eta M^{k+1})^{\frac{q}{p}}dt\\
&=& \displaystyle\sum_{k=0}^{\infty}\eta^{q}M^{q(k+1)}(1-M^{-q})\omega_{h}(\eta M^{k+1})^{\frac{q}{p}}=C_{1}^{-1}S,
\end{eqnarray*}
where $C_{1}=(\eta^{q}(1-M^{-q}))^{-1}>0$. Hence, $S<\infty$. Conversely, if $S<\infty$ then analogously above it follows that
\begin{eqnarray*}
\|h\|_{L^{p,q}_{\omega}(E)}^{q}\leq C_{2}(\omega(E)^{\frac{q}{p}}+S)<\infty,
\end{eqnarray*}
for $C_{2}=(\eta M)^{q}>0$. Thus, $h\in L^{p,q}_{\omega}(E)$. The estimate in \eqref{characterization1} it follows of estimates above where $C=\max\{C_{1},C_{2}\}$.

Now, in the case $p\in(0,\infty)$ and $q=\infty$, we have for all $j\in\mathbb{N}$ that,
\begin{eqnarray*}
M^{j}\omega_{h}(\eta M^{j})^{\frac{1}{p}}\leq \frac{1}{\eta}\|h\|_{L^{p,\infty}_{\omega}(E)}
\end{eqnarray*}
consequently $\|h\|_{L^{p,\infty}_{\omega}(E)}\ge (C'_{1})^{-1} T$, for $C'_{1}=\eta^{-1}$. On the other hand, on the assumption that $M>1$, we can write $(0,\infty)=(0,\eta M]\dot{\cup}\left(\displaystyle\dot{\bigcup}_{k=1}^{\infty}(\eta M^{k},\eta M^{k+1}]\right)$. Hence, given $t>0$, we have that $t\in (0,\eta M]$ or $t\in (\eta M^{k},\eta M^{+1})$ for some $k\in\mathbb{N}$. In the first case,
\begin{eqnarray*}
t\omega_{h}(t)^{\frac{1}{p}}\leq (\eta M)\omega(E)^{\frac{1}{p}}
\end{eqnarray*}
and in other case,
\begin{eqnarray*}
t\omega_{h}(t)^{\frac{1}{p}}\leq (\eta M^{k+1})\omega_{h}(\eta M^{k})^{\frac{1}{p}}.
\end{eqnarray*}
Thus, $t\omega_{h}(t)^{\frac{1}{p}}\leq C'_{2}(\omega(E)^{\frac{1}{p}}+T)$, for $C'_{2}=\max\{M,\eta M\}>0$, which implies in estimate  $$\|h\|_{L^{p,\infty}_{\omega}(E)}\leq C'_{2}(\omega(E)^{\frac{1}{p}}+T).$$ Finally, taking $\overline{C}=\max\{C'_{1},C'_{2}\}$ it follows  \eqref{characterization2}. 
\end{proof}

In this part we will present a key tool that provides a compactness method that guarantees the possibility of weakening the convexity hypothesis on the ruling operator $F$ throughout the work. In fact, the following result tells us that if our equation is close enough to the homogeneous equation with constant coefficients, then our solution will be close enough to a solution of the homogeneous equation with frozen coefficients. At the heart of our techniques is the notion of the recession operator (see Definition \ref{DefR}). More precisely, we have the following lemma, the proof of which can be found in \cite[Lemma 2.12]{Bessa},


\begin{lemma}[{\bf Approximation}] \label{Approx}
Let $n \le p < \infty$, $0 \le \nu \le 1$ and assume that $(E1),(E2),(E3)$ and $(E5)$ are in force. Then, for every $\delta >0$, $\varphi \in C(\partial \mathrm{B}_1(0^{\prime},\nu))$ with $\|\varphi\|_{L^{\infty}(\partial \mathrm{B}_1(0^{\prime},\nu))} \le \mathfrak{c}_1$ and $g \in C^{1,\alpha}(\overline{\mathrm{T}}_2)$ with $0 < \alpha < 1$  and $\|g\|_{C^{1,\alpha}(\overline{\mathrm{T}}_2)} \le \mathfrak{c}_2$ for some $\mathfrak{c}_2 >0$ there exist positive constants $\epsilon =\epsilon(\delta,n, \mu_0, p, \lambda, \Lambda, \gamma,\mathfrak{c}_1, \mathfrak{c}_2) < 1$ and $\tau_0 = \tau_0(\delta, n, \lambda,\Lambda, \mu_0, \mathfrak{c}_1, \mathfrak{c}_2) >0$ such that, if
$$
\max\left\{ |F_{\tau}(\mathrm{X},x) - F^{\sharp}(\mathrm{X},x)|, \, \|\psi_{F^{\sharp}}(\cdot,0)\Vert_{L^{p}(\mathrm{B}^{+}_{r})},\,\|f\|_{L^{p}(\mathrm{B}^+_{r})}  \right\} \le \epsilon \quad \textrm{and} \quad \tau \le \tau_0
$$
then, any two $L^p$-viscosity solutions $u$ (normalized, i.e., $\|u\|_{L^{\infty}(\mathrm{B}^{+}_r(0^{\prime},\nu))}\le 1$) and $\mathfrak{h}$ of
$$
\left\{
\begin{array}{rclcl}
F_{\tau}(D^2u,x) &=& f(x)& \mbox{in} & \mathrm{B}^{+}_r(0^{\prime},\nu) \cap \mathbb{R}^{n}_+ \\
\mathfrak{B}(Du,u,x)&=& g & \mbox{on} &  \mathrm{B}_r(0^{\prime},\nu) \cap \mathrm{T}_r\\
u(x) &=& \varphi(x) &\mbox{on}& \overline{\partial \mathrm{B}_r(0^{\prime},\nu) \cap \mathbb{R}^n_+}
\end{array}
\right.
$$
and
$$
\left\{
\begin{array}{rclcl}
F^{\sharp}(D^2 \mathfrak{h},0) &=& 0& \mbox{in} & \mathrm{B}^{+}_{\frac{7}{8}r}(0^{\prime},\nu) \cap \mathbb{R}^n_+ \\
\mathfrak{B}(D\mathfrak{h},\mathfrak{h},x) &=& g & \mbox{on} &  \mathrm{B}_{\frac{7}{8} r}(0^{\prime},\nu) \cap \mathrm{T}_r\\
\mathfrak{h}(x) &=& u(x) &\mbox{on}& \overline{\partial \mathrm{B}_{\frac{7}{8}r}(0^{\prime}, \nu) \cap \mathbb{R}^n_+}		
\end{array}
\right.
$$
satisfy
$$
\|u-\mathfrak{h}\|_{L^{\infty}(\mathrm{B}^{+}_{\frac{7}{8}r}(0^{\prime},\nu))} \le \delta.
$$
\end{lemma}

With these approximation lemma we obtain the following lemma (cf. \cite[Lemma 3.6]{Bessa}):

\begin{lemma} \label{propdecaimento}
Given $\epsilon_0  \in (0, 1)$ and let $u$ be a normalized viscosity solution to
$$
\left\{
\begin{array}{rclcl}
F_{\tau}(D^2u,x) &=& f(x) & \mbox{in} & \mathrm{B}^+_{14\sqrt{n}},\\
\mathfrak{B}(Du,u,x)&=& g & \mbox{on} & \mathrm{T}_{14 \sqrt{n}}.
\end{array}
\right.
$$
Assume that (E1)-(E4) hold and extend $f$ by zero outside $\mathrm{B}^+_{14\sqrt{n}}$. For $x \in \mathrm{B}_{14 \sqrt{n}}$, let
	$$
	\max \left\{\tau, \|f\|_{L^n(\mathrm{B}_{14 \sqrt{n}})} \right\} \le \tilde{\epsilon}
	$$
	for some $\tilde{\epsilon} >0$ depending only $n, \epsilon_0, \lambda, \Lambda, \mu_0$ and $\alpha$. Then, for $k \in \mathbb{N}\setminus \{0\}$ we define
	\begin{eqnarray*}
		\mathcal{A} & \defeq & \mathcal{A}_{\mathrm{M}^{k+1}}(u, \mathrm{B}^+_{14 \sqrt{n}}) \cap \left(\mathcal{Q}^{n-1}_1 \times (0,1)\right)\\
		\mathcal{B} &\defeq & \left(\mathcal{A}_{\mathrm{M}^k}(u, \mathrm{B}^+_{14\sqrt{n}}) \cap \left(\mathcal{Q}^{n-1}_1 \times (0,1)\right)\right)\cup \left\{x \in \mathcal{Q}^{n-1}_1 \times (0,1); \mathcal{M}(f^n) \ge (\mathrm{C}_0\mathrm{M}^k)^n \right\},
	\end{eqnarray*}
where $\mathrm{M} = \mathrm{M}(n, \mathrm{C}_0)>1$. Then,
$$
|\mathcal{A}| \le \epsilon_0|\mathcal{B}|.
	$$
\end{lemma}

From the strong doubling in Lemma \ref{propriedadesdospesos}, we obtain a decay of the measure of the sets $A_{k}$ of Lemma \ref{propdecaimento} with respect to the measure $\omega$. This fact is summarized by the following proposition:

\begin{corollary}\label{corolariodecaimento}
Under the same conditions of Lemma \ref{propdecaimento} and assuming  that $\omega\in \mathfrak{A}_{p}$ for some $1<p<\infty$, fix $\epsilon_{0}\in(0,1)$. Now, for $k \geq 0$ we define
\begin{eqnarray*}
A^{k}\defeq  \mathcal{A}_{\mathrm{M}^k}(u, \mathrm{B}^+_{14\sqrt{n}}) \cap \left(\mathcal{Q}^{n-1}_1 \times (0,1)\right) \  \mbox{and} \ 
B^{k} \defeq  \left\{ x \in \left(\mathcal{Q}^{n-1}_1 \times (0,1)\right); \mathcal{M}(f^n)(x) \ge (\mathrm{C}_0\mathrm{M}^k)^n\right\},
\end{eqnarray*}
where $\mathrm{C}_{0}$ and $\mathrm{M}$ are the constants in Lemma \ref{propdecaimento}. Then, for any $k\geq 0$,
\begin{eqnarray*}
\omega\left(A^{k}\right)\leq \epsilon_{0}^{k}\omega(A^{0})+\sum_{i=1}^{k-1}\epsilon_{0}^{k-i}\omega(B^{i}).
\end{eqnarray*}
\end{corollary}
\begin{proof}
In effect, fix $\epsilon_{0}\in(0,1)$. By the Lemma \ref{propdecaimento} to the constant $\tilde{\epsilon}_{0}=\left(\frac{\epsilon_{0}}{c_{1}}\right)^{\frac{ 1}{\theta}}$, where $\theta$ and $k_{1}$ are the constants of Lemma \ref{propriedadesdospesos}, we obtain the following estimate
\begin{eqnarray*}
|A^{k+1}|\leq \tilde{\epsilon_{0}}|A^{k}\cup B^{k}|
\end{eqnarray*}
and consequently by strong doubling (item (c) of Lemma \ref{propriedadesdospesos}) we obtain that
\begin{eqnarray*}
\omega(A^{k+1})\leq c_{1}\left( \frac{|A^{k+1}|}{|A^{k}\cup B^{k}|}\right)^{\theta}\omega(A^{k}\cup B^{k})= \epsilon_{0}\omega(A^{k}\cup B^{k})\le \epsilon_{0}\omega(A^{k})+\epsilon_{0}\omega(B^{k}), \forall k\geq 0.
\end{eqnarray*}
By iterating over these estimates, it follows the desired.
\end{proof}

\section{ Proof of main theorem } \label{section3}

\hspace{0.4cm}For the proof of the Main Theorem, we will need to study the regularity of solutions to the problem
\begin{equation} \label{equation1}
\left\{
\begin{array}{rclcl}
F(D^2u,Du,u,x) &=& f(x) & \mbox{in} & \mathrm{B}^+_1,\\
\mathfrak{B}(Du,u,x)&=& g & \mbox{on} & \mathrm{T}_1.
\end{array}
\right.
\end{equation}

Initially we have for $C^{0}$-viscosity solutions of \eqref{equation1} without the dependence of $Du$ and $u$ the following result:

\begin{proposition}\label{T-flat}
Consider $(p,q)\in (n,\infty)\times(0,\infty]$, $f \in L^{p,q}_{\omega}(\mathrm{B}^+_1)\cap C^{0}(\mathrm{B}^{+}_{1})$ and $\omega$ in $\mathfrak{A}_{\frac{p}{n}}$ be an weight. Let $u$ be a normalized $C^{0}$-viscosity solution of
\begin{equation*} \label{mens}
\left\{
\begin{array}{rclcl}
F(D^2u, x) &=& f(x) & \mbox{in} & \mathrm{B}^+_1,\\
\mathfrak{B}(Du,u,x)&=& g & \mbox{on} & \mathrm{T}_1.
\end{array}
\right.
\end{equation*}
 Assume that assumptions (E1)-(E5) are in force. Then $D^{2}u \in L^{p,q}_{\omega}\left(\mathrm{B}^+_{\frac{1}{2}}\right)$ and
$$
\|D^{2}u\|_{L^{p,q}_{\omega}\left(\mathrm{B}^+_{\frac{1}{2}}\right)} \le \mathrm{C} \cdot \left( \|u\|_{L^{\infty}(\mathrm{B}^+_1)} + \|f\|_{L^{p,q}_{\omega}(\mathrm{B}^+_1)}+\Vert g\Vert_{C^{1,\alpha}(\overline{\mathrm{T}_{1}})}\right),
$$
where $\mathrm{C}=\mathrm{C}(n,\lambda,\Lambda,p,q,[\omega]_{\frac{p}{n}}, \Vert \beta\Vert_{C^{1,\alpha}(\mathrm{T}_{1})}, \Vert \gamma\Vert_{C^{1,\alpha}(\mathrm{T}_{1})}, \alpha,r_{0},\theta_0,\mu_{0})>0$.
\end{proposition}

\begin{proof}
Fix $x_0 \in \mathrm{B}^{+}_{1/2}\cup T_{\frac{1}{2}}$. When $x_0 \in \mathrm{T}_{\frac{1}{2}}$, we choosen $0 < r < \frac{1-|x_0|}{14\sqrt{n}}$ and define
$$
\kappa \defeq \frac{\tilde{\epsilon} r}{\tilde{\epsilon} r^{-1} \|u\|_{L^{\infty}(\mathrm{B}^+_{14r\sqrt{n}}(x_0))}  + \mathrm{C}_{p,q}\|f\|_{L^{p,q}_{\omega}(\mathrm{B}^+_{14r\sqrt{n}}(x_0))}+\tilde{\epsilon} r^{-1}\Vert g\Vert_{C^{1,\alpha}(\mathrm{T}_{14r\sqrt{n}}(x_{0}))}}
$$
where the constant $\tilde{\epsilon}>0$ is the constant  of the Lemma \ref{propdecaimento} for $\epsilon_{0}>0$ which will be determined a posteriori and $\mathrm{C}_{p,q}$ is the constant of the Lemma \ref{Incluaolorentzlebesgue}. Due to the limitation of $r$, the function $\tilde{u}(y) \defeq \frac{\kappa}{r^{2}}u(x_0+ry), y\in B^{+}_{14\sqrt{n}}$ is well defined  and is a normalized $C^{0}$-viscosity solution to
$$
\left\{
\begin{array}{rclcl}
\tilde{F}(D^2 \tilde{u},x) &=& \tilde{f}(x) & \mbox{in} & \mathrm{B}^+_{14\sqrt{n}},\\
\tilde{\beta}\cdot D\tilde{u}+\tilde{\gamma}\tilde{u}
&=&\tilde{g}  & \mbox{on} & \mathrm{T}_{14\sqrt{n}} .
\end{array}
\right.
$$
where
$$
\left\{
\begin{array}{rcl}
\tilde{F}(\mathrm{X}, y) &\defeq& \kappa F\left(\frac{1}{\kappa} \mathrm{X},x_{0}+ry\right) \\
\tilde{f}(y) &\defeq& \kappa f(x_0+ry)\\
\tilde{\beta}(y) &\defeq&  \beta(x_0+ry)\\
\tilde{\gamma}(y) &\defeq& r\gamma(x_{0}+ry)\\
\tilde{g}(y)&\defeq& \frac{\kappa}{r}g(x_{0}+ry)\\
\tilde{\omega}(y)&\defeq&\omega(x_{0}+ry).
\end{array}
\right.
$$
Hence, $\tilde{F}$ fulfills the conditions (E1)-(E5) and $\tilde{\omega}\in \mathfrak{A}_{\frac{p}{n}}$ (since $\omega\in \mathfrak{A}_{\frac{p}{n}}$). Moreover, Lemma $\ref{Incluaolorentzlebesgue}$ ensures that
\begin{eqnarray} \label{(16)}
\|\tilde{f}\|_{L^{n}\left(\mathrm{B}^+_{14\sqrt{n}}\right)} &=& \frac{\kappa}{r}\|f\|_{L^{n}(\mathrm{B}^{+}_{14r\sqrt{n}}(x_{0}))}\leq \frac{\kappa\mathrm{C}_{p,q}}{r}\|f\|_{L^{p,q}_{\omega}\left(\mathrm{B}^+_{14r\sqrt{n}}\right)} \le \tilde{\epsilon}.
\end{eqnarray}
Thus, by Corollary \ref{corolariodecaimento}, we can conclude that
\begin{eqnarray}\label{3.2}
\tilde{\omega}(A^{k})\leq \tilde{\epsilon}_{0}^{k}\tilde{\omega}(A_{0})+\displaystyle\sum_{i=1}^{k-1}\tilde{\epsilon}_{0}^{i}\tilde{\omega}(B^{k-i}),
\end{eqnarray}
where for $k \geq 0$ we define
\begin{eqnarray*}
A^{k}&\defeq & \mathcal{A}_{\mathrm{M}^k}(\tilde{u}, \mathrm{B}^+_{14\sqrt{n}}) \cap \left(\mathcal{Q}^{n-1}_1 \times (0,1)\right)\\
B^{k} &\defeq&  \left\{ x \in \left(\mathcal{Q}^{n-1}_1 \times (0,1)\right); \,\, \mathcal{M}(\tilde{f}^n)(x) \ge (\mathrm{C}_0\mathrm{M}^k)^n\right\}.
\end{eqnarray*}
On the other hand, the Lemma \ref{Hardylpq} ensures that
\begin{eqnarray*}
\|\mathcal{M}(\tilde{f}^{n})\|_{L^{\frac{p}{n},\frac{q}{n}}_{\tilde{\omega}}(\mathcal{Q}_{1}^{n-1}\times (0,1))}\leq C\| \tilde{f^{n}}\|_{L^{\frac{p}{n},\frac{q}{n}}_{\tilde{\omega}}(\mathcal{Q}^{n-1}_{1}\times (0,1))}< \infty,
\end{eqnarray*}
since, $f\in L^{p,q}_{\omega}(B^{+}_{1})$.
Therefore, using \eqref{3.2}  and Proposition \ref{P1} we obtain
\begin{eqnarray*}
\sum_{k=1}^{\infty} \mathrm{M}^{qk} \tilde{\omega}(A^{k})^{\frac{q}{p}} &\le& C(n,p,q)\left( \sum_{k=1}^{\infty}(M^{q}\tilde{\epsilon}_{0}^{\frac{q}{p}})^{k}\tilde{\omega}(A_{0})^{\frac{q}{p}}+\sum_{k=1}^{\infty} M^{qk}\sum_{i=1}^{k}\tilde{\epsilon}^{\frac{q}{p}i}_{0}\tilde{\omega}^{\frac{q}{p}}(B^{k-i})\right)\\
&\leq& C\left(\sum_{k=1}^{\infty}(M^{q}\tilde{\epsilon}^{\frac{q}{p}}_{0})^{k}\tilde{\omega}^{\frac{q}{p}}(\mathcal{Q}^{n-1}_{1}\times (0,1))+\sum_{k=1}^{\infty} \sum_{i=1}^{k-1}(M^{q}\tilde{\epsilon}^{\frac{q}{p}}_{0})^{i} M^{q(k-i)}\tilde{\omega}^{\frac{q}{p}}(B^{k-i})\right)\\
&\leq&C \sum_{k=1}^{\infty}(M^{q}\tilde{\epsilon}^{\frac{q}{p}}_{0})^{k}\left(\tilde{\omega}(\mathcal{Q}^{n-1}_{1}\times(0,1))+\sum_{i=1}^{\infty}M^{qi}\tilde{\omega}(B^{i})^{\frac{q}{p}}\right)\leq C(n,p,q,[\omega]_{\frac{p}{n}}),
\end{eqnarray*}
chosen $\epsilon_{0}$ small enough in such a way that $M^{q}\tilde{\epsilon}^{\frac{q}{p}}_{0}<1$. This estimate together with the Lemma \ref{P1} implies that $\Theta(\tilde{u},B^{+}_{\frac{1}{2}})\in L^{p,q}_{\tilde{\omega}}(B^{+}_{\frac{1}{2}})$ and consequently by Lemma \ref{caracterizationofhessian}, $\|D^{2}\tilde{u}\|_{L^{p,q}_{\tilde{\omega}}(B^{+}_{\frac{1}{2}})}\le C$, for positive constant $C$ depending only on $n$, $\lambda$, $\Lambda$, $p$, $q$, $[\omega]_{\frac{p}{n}}$, $\|\beta\|_{C^{1,\alpha}(T_{1})}$, $\|\gamma\|_{C^{1,\alpha}(\mathrm{T}_{1})}$, $\mu_0$, $r_{0}$ and $\theta_0$ (constants of condition (E3)). Rescaling $u$, we get   
\begin{eqnarray*}
\|D^{2}u\|_{L^{p,q}_{\omega}(B^{+}_{\frac{r}{2}}(x_{0}))}\leq C(\|u\|_{L^{\infty}(B^{+}_{1})}+\|f\|_{L^{p,q}_{\omega}(B^{+}_{1})}+\|g\|_{C^{1,\alpha}(T_{1})}),
\end{eqnarray*}
where $C>0$ depends only on $n$, $\lambda$, $\Lambda$, $p$, $q$, $[\omega]_{\frac{p}{n}}$,  $\|\beta\|_{C^{1,\alpha}(T_{1})}$,$\|\gamma\|_{C^{1,\alpha}(\mathrm{T}_{1})}$, $\mu_{0}$, $r_{0}$, $\theta_0$ and $r$.

On the other hand, if $x_0 \in \mathrm{B}^+_{1/2}$, then using the lemmas 4.2 and 5.3 in \cite{PT}(with minor modifications) instead of lemmas \ref{Approx} and  \ref{propdecaimento}, we proceed entirely analogously as above to obtain for $r$ small enough that 
\begin{eqnarray*}
\|D^{2}u\|_{L^{p,q}_{\omega}(B_{\frac{r}{2}}(x_{0}))}\leq C(\|u\|_{L^{\infty}(B^{+}_{1})}+\|f\|_{L^{p,q}_{\omega}(B^{+}_{1})}+\|g\|_{C^{1,\alpha}(T_{1})}).
\end{eqnarray*}
Finally, by combining interior and boundary estimates, we obtain the desired results by using a standard covering argument. This finishes proof of  case $q\in(0,\infty)$ of Proposition. The case $q=\infty$ follows proceeds with analogous ideas above. This ends the proof of the Proposition. 
	
\end{proof}


With this proposition in hand, we can find weighted Lorentz-Sobolev estimates for $C^{0}$-viscosity solutions of \eqref{equation1} by

\begin{proposition}\label{solution4entradas}
Let $u$ be a bounded $C^{0}$-viscosity solution of \eqref{equation1}. Asssume the  $f\in L^{p,q}_{\omega}(B^{+}_{1})$ $((p,q)\in (n,+\infty)\times (0,+\infty])$, $\omega\in \mathfrak{A}_{\frac{p}{n}}$ and  (E1)-(E5) are in force. Then, there exists a constant $\mathrm{C}>0$ depending on  $n$, $\lambda$, $\Lambda$, $p$, $q$, $[\omega]_{\frac{p}{n}}$, $\Vert \beta\Vert_{C^{1,\alpha}(\mathrm{T}_{1})}$, $\Vert \gamma\Vert_{C^{1,\alpha}(\overline{\mathrm{T}_{1}})}$, $\alpha$, $r_{0}$, $\theta_0$ and $\mu_{0}$, such that  $u \in W^{2}L^{p,q}_{\omega} \left(B^+_{\frac{1}{2}}\right)$ and
$$	\|u\|_{W^{2}L^{p,q}_{\omega}\left(\mathrm{B}^+_{\frac{1}{2}}\right)} \le \mathrm{C} \cdot \left( \|u\|_{L^{\infty}(\mathrm{B}^{+}_1)} + \|f\|_{L^{p,q}_{\omega}(\mathrm{B}^{+}_1)}+\Vert g\Vert_{C^{1,\alpha}(\mathrm{T}_{1})}\right).
$$
\end{proposition}

\begin{proof}
Notice that $u$ is also a viscosity solution of
\begin{equation*}
\left\{
\begin{array}{rclcl}
\tilde{F}(D^2u, x) &=& \tilde{f}(x) & \mbox{in} & \mathrm{B}^+_1,\\
\mathfrak{B}(Du,u,x)&=& g(x) & \mbox{on} & \mathrm{T}_1 .
\end{array}
\right.
\end{equation*}
where $\tilde{F}(\mathrm{X},x)  \defeq  F(\mathrm{X},0,0,x)$ and $\tilde{f}$ is a function satisfying
\begin{eqnarray*}
|\tilde{f}|\leq \sigma|Du|+\xi|u|+|f|.
\end{eqnarray*}
Thus, we can apply Proposition \ref{T-flat} and conclude that
\begin{eqnarray}\label{mens2}
\|D^{2}u\|_{L^{p,q}_{\omega}\left(\mathrm{B}^{+}_{\frac{1}{2}}\right)} \le \mathrm{C} \cdot \left( \|u\|_{L^{\infty}(\mathrm{B}^{+}_{1})} + \|\tilde{f}\|_{L^{p,q}_{\omega}(\mathrm{B}^+_1)}+\Vert g\Vert_{C^{1,\alpha}(\mathrm{T}_{1})}\right).
\end{eqnarray}
Now, by proceeding similarly as in \cite[Lemma 3.2]{ZhangZhenglorentz} for obtain gradient estimate,
\begin{equation}\label{mens3}
\Vert Du\Vert_{L^{p,q}_{\omega}(\mathrm{B}_{\frac{1}{2}}^{+})}\leq \mathrm{C}\cdot (\Vert u\Vert_{L^{\infty}(\mathrm{B}^{+}_{1})}+\Vert f\Vert_{L^{p,q}_{\omega}(\mathrm{B}^{+}_{1})}).
\end{equation}
Finally, combining the estimates (\ref{mens2}) and (\ref{mens3})  and the fact that $u\in L^{p,q}_{\omega}(B^{+}_{\frac{1}{2}})$ with estimate $\|u\|_{L^{p,q}_{\omega}(B^{+}_{\frac{1}{2}})}\leq \mathrm{C}\|u\|_{L^{\infty}(B^{+}_{1})}$ we complete the desired estimate.
\end{proof}

Consequently, we have by classical arguments (cf. \cite{Bessa} and \cite{Winter}) the extension of the Proposition \ref{solution4entradas} to $L^{\tilde{p}}$-viscosity solutions.

\begin{corollary}\label{Cor}
Let $u$ be a bounded $L^{\tilde{p}}-$viscosity solution of \eqref{equation1}, where $\beta,\gamma, g \in C^{1,\alpha}(\mathrm{T}_1)$ with $\beta \cdot \textbf{n} \ge \mu_0$ for some $\mu_0 >0$, $\gamma \le 0$, $f \in L^{p,q}_{\omega}(\mathrm{B}^+_1)$, for $(p,q)\in(n,\infty)\times (0,\infty]$, $\omega\in \mathfrak{A}_{\frac{p}{n}}$ with condition $\tilde{p}\in[n,p)$. Further, assume that  $F^{\sharp}$ satisfies $(E4)-(E5)$ and $F$ fulfills the condition $(SC)$. Then, there exist positive constants $\beta_0=\beta_0(n,\lambda,\Lambda,p,q,\tilde{p})$, $r_0 = r_0(n, \lambda, \Lambda, p,q,\tilde{p})$ and $\mathrm{C}=\mathrm{C}(n,\lambda,\Lambda,p,q,\tilde{p},[\omega]_{\frac{p}{n}},\theta_0,\|\beta\|_{C^{1,\alpha}(\mathrm{T}_{1})},\|\gamma\|_{C^{1,\alpha}(\mathrm{T}_{1})},r_0)>0$, such that if
$$
\left(\intav{\mathrm{B}_r(x_0) \cap \mathrm{B}^+_1} \psi_{F^{\sharp}}(x, x_0)^p dx\right)^{\frac{1}{p}} \le \psi_0
$$
for any $x_0 \in \mathrm{B}^+_1$ and $r \in (0, r_0)$, then $u \in W^{2}L^{p,q}_{\omega}\left(\mathrm{B}^+_{\frac{1}{2}}\right)$ and
$$
\|u\|_{W^{2}L^{p,q}_{\omega}\left(\mathrm{B}^+_{\frac{1}{2}}\right)} \le \mathrm{C} \cdot\left( \|u\|_{L^{\infty}(\mathrm{B}^+_1)} +\|f\|_{L^{p,q}_{\omega}(\mathrm{B}^+_1)}+\Vert g\Vert_{C^{1,\alpha}(\mathrm{T}_{1})}\right).
$$
\end{corollary}

\begin{proof}
It is enough to prove the result for equations without dependence on $Du$ and $u$ (see \cite[Theorem 4.3]{Winter} for details). We will approximate $f$ in $L^{p,q}_{\omega}$ by functions $f_j \in C^{\infty}(\overline{\mathrm{B}^+_1}) \cap L^{p,q}_{\omega}(\mathrm{B}^+_1)$ such that $f_{j}\to f$ in $L^{p,q}_{\omega}(\mathrm{B}^{+}_{1})$ and also approximate $g$ by a sequence $(g_{j})$ in $C^{1,\alpha}(\mathrm{T}_{1})$ with the property that $g_{j}\to g$ in $C^{1,\alpha}(\mathrm{T}_{1})$. By Theorem of Existence and Uniqueness  \ref{Existencia}, there exists a sequence of functions $ u_{j}\in C^{0}(\overline{\mathrm{B}^{+}_{1}})$, which they are viscosity solutions of following family of PDEs:
$$
\left\{
\begin{array}{rclcl}
F(D^2u_j, x) &=& f_j(x) & \mbox{in} & \mathrm{B}^+_1,\\
\mathfrak{B}(Du_{j},u_{j},x)&=& g_{j} & \mbox{on} & \mathrm{T}_{1}\\
u_{j}(x)&=&u(x) & \mbox{on} & \partial \mathrm{B}^{+}_{1}\setminus \mathrm{T}_{1}.
\end{array}
\right.
$$
Therefore, the assumptions of the Proposition $\ref{solution4entradas}$   are in force. By these result, we have that
$$
\|u_j\|_{W^{2}L^{p,q}_{\omega}\left(\mathrm{B}^{+}_{\frac{1}{2}}\right)} \le \mathrm{C}.\left( \|u_j\|_{L^{\infty}(\mathrm{B}^+_1)} + \|f_j\|_{L^{p,q}_{\omega}(\mathrm{B}^+_1)}+\Vert g_{j}\Vert_{C^{1,\alpha}(\mathrm{T}_{1})}\right),
$$
for a universal constant $\mathrm{C}>0$. Furthermore, a standard covering argument yields $u_j \in W^{2}L^{p,q}_{\omega}(\mathrm{B}^+_1)$. From Lemma $\ref{ABP-fullversion}$, $(u_j)_{j \in \mathbb{N}}$ is uniformly bounded in $W^{2}L^{p,q}_{\omega}(\overline{\mathrm{B}^+_{\rho}})$ for $\rho  \in (0, 1)$.

Once again, we can by ABP Maximum Principle \ref{ABP-fullversion} and Lemma \ref{Incluaolorentzlebesgue} we obtain,
\begin{eqnarray*}
\|u_j-u_k\|_{L^{\infty}\left(\mathrm{B}^+_{1}\right)} &\le& \mathrm{C}(n,\lambda,\Lambda, \mu_0)(\|f_j-f_k\|_{L^{p}(\mathrm{B}^+_1)}+\Vert g_{j}-g_{k}\Vert_{C^{1,\alpha}(\mathrm{T}_{1})})\\
&\leq&C(n,\lambda,\Lambda,\mu_{0},[\omega]_{\frac{p}{n}},p, \tilde{p})(\|f_{j}-f_{k}\|_{L^{p,q}_{\omega}(\mathrm{B}^{+}_{1})}+\Vert g_{j}-g_{k}\Vert_{C^{1,\alpha}(\mathrm{T}_{1})}).
\end{eqnarray*}
Thus, $u_j \to u_{\infty} \quad \textrm{in} \quad C^0(\overline{\mathrm{B}^+_1})$. Moreover, since $(u_j)_{j \in \mathbb{N}}$ is bounded in $W^{2}L^{p,q}_{\omega}\left(\mathrm{B}^+_{\frac{1}{2}}\right)$  we obtain $u_j \rightharpoonup u_{\infty}$ weakly in $W^{2}L^{p,q}_{\infty}\left(\mathrm{B}^+_{\frac{1}{2}}\right)$.
Thus,
$$
\|u_{\infty}\|_{W^{2}L^{p,q}_{\omega}(\mathrm{B}^{+}_{1/2})} \le \mathrm{C} \cdot\left( \|u_{\infty}\|_{L^{\infty}(\mathrm{B}^+_1)} + \|f\|_{L^{p,q}_{\omega}(\mathrm{B}^+_1)}+\Vert g\Vert_{C^{1,\alpha}(\mathrm{T}_{1})}\right).
$$
Finally, stability results (see Lemma \ref{Est}) ensure that $u_{\infty}$ is an $L^{\tilde{p}}-$viscosity solution to
$$
\left\{
\begin{array}{rclcl}
F(D^2u_{\infty}, x) &=& f(x)& \mbox{in} & \mathrm{B}^+_1,\\
\mathfrak{B}(Du_{\infty},u_{\infty},x)&=& g & \mbox{on} & \mathrm{T}_1 \\
u_{\infty}(x)&=& u(x) & \mbox{on} & \partial \mathrm{B}^{+}_{1}\setminus \mathrm{T}_{1}.
\end{array}
\right.
$$
Thus, $w \defeq u_{\infty}-u$ fulfills
$$
\left\{
\begin{array}{rclcl}
w\in S(\frac{\lambda}{n}, \Lambda,0) & \mbox{in} & \mathrm{B}^+_1,\\
\mathfrak{B}(Dw,w,x)= 0 & \mbox{on} & \mathrm{T}_1 \\
w= 0 & \mbox{on} & \partial \mathrm{B}^{+}_{1}\setminus \mathrm{T}_{1}.
\end{array}
\right.
$$
which by Lemma $\ref{ABP-fullversion}$ we conclude that $w=0$ in $\overline{\mathrm{B}^{+}_{1}}\setminus \mathrm{T}_{1}$. By continuity, $w=0$ in $\overline{\mathrm{B}^{+}_{1}}$ which finishes the proof.
\end{proof}

\bigskip

Finally, we are now in position to prove the Theorem \ref{T1}.

\begin{proof}[{\bf Proof of Theorem \ref{T1}}]
Based on standard reasoning (cf.\cite{Winter} and  \cite{ZhangZhengmorrey}) it is important to highlight that is always possible to perform a change of variables in order to flatten the boundary, since $\partial \Omega \in C^{2,\alpha}$. For this end, consider $x_0 \in \partial \Omega$, by $\partial \Omega \in C^{2,\alpha}$ there exists a neighborhood of $x_0$, which we will label $\mathcal{V}(x_0)$ and a $C^{2, \alpha}-$diffeomorfism $\Phi: \mathcal{V}(x_0) \to \mathrm{B}_1(0)$ such that $    \Phi(x_0) = 0 \quad \mbox{and} \quad \Phi(\Omega \cap \mathcal{V}(x_0)) = \mathrm{B}^{+}_1$. Now, we define $\tilde{u} \defeq u \circ \Phi^{-1} \in C^0(\mathrm{B}^+_1)$. Next, observe that $\tilde{u}$ is an $L^{\tilde{p}}-$viscosity solution to
$$
\left\{
\begin{array}{rclcl}
 \tilde{F}(D^2 \tilde{u}, D \tilde{u}, \tilde{u}, y) &=& \tilde{f}(x) & \mbox{in} & \mathrm{B}^+_1,\\
\tilde{\beta}\cdot D\tilde{u}+\tilde{\gamma}\tilde{u} & = & \tilde{g} &\mbox{on} & \mathrm{T}_1.
\end{array}
\right.
$$
where for $y=\Phi^{-1}(x)$ we have
$$
\left\{
\begin{array}{rcl}
\tilde{F}(\mathrm{X}, \varsigma, \eta, y) &=& F\left(D \Phi^t(y) \cdot \mathrm{X} \cdot D \Phi(y) + \varsigma D^2\Phi, \varsigma D\Phi(y), \eta, y\right)\\
\tilde{f}(x) & \defeq & f \circ \Phi^{-1}(x)\\
\tilde{\beta}(x) & \defeq & (\beta \circ \Phi^{-1}) \cdot (D \Phi \circ \Phi^{-1})^{t}\\
\tilde{\gamma}(x) & \defeq & (\gamma \circ \Phi^{-1}) \cdot (D \Phi \circ \Phi^{-1})^{t}\\
\tilde{g}(x) & \defeq & g \circ \Phi^{-1}\\
\tilde{\omega}(x)&=&\omega\circ \Phi^{-1}(x).
\end{array}
\right.
$$
Furthermore, note that $\tilde{F}$ is a uniformly elliptic operator with ellipticity constants $\lambda \mathrm{C}(\Phi)$, $\Lambda \mathrm{C}(\Phi)$.
Thus,
$$
\tilde{F}^{\sharp}(\mathrm{X}, \varsigma, \eta, x) = F^{\sharp}\left(D\Phi^t(\Phi^{-1}(x)) \cdot \mathrm{X}\cdot D\Phi(\Phi^{-1}(x)) +  \varsigma D^2 \Phi(\Phi^{-1}(x)), 0,0, \Phi^{-1}(x)\right).
$$
In consequence, we conclude that $\psi_{\tilde{F}^{\sharp}}(x,x_0) \le \mathrm{C}(\Phi) \psi_{F^{\sharp}}(x,x_0)$, which ensures that $\tilde{F}$ falls into the assumptions of Corollary \ref{Cor}. Thus, $u\in W^{2}L^{p,q}_{\omega}(\Omega)$ and 
\begin{eqnarray}\label{estimativa3.6}
\|u\|_{W^{2}L^{p,q}_{\omega}(\Omega)}\leq C(\|u\|_{L^{\infty}(\Omega)}+\|f\|_{L^{p,q}_{\omega}(\Omega)}+\|g\|_{C^{1,\alpha}(\partial\Omega)}).
\end{eqnarray}
Now we proof the estimate \eqref{2}, suppose for absurdity that such an estimate is false. Then there exists sequences $(u_{j})_{j\in\mathbb{N}}$, $(f_{j})_{j\in\mathbb{N}}$ and $(g_{j})_{j\in\mathbb{N}}$ such that $u_{k}$ is $L^{\tilde{p}}$-viscosity solution of 
$$
\left\{
\begin{array}{rclcl}
F(D^2 u_{j}, Du_{j}, u_{j}, x) &=& f_{j}(x) & \mbox{in} & \Omega,\\
\mathfrak{B}(Du_{j},u_{j},x) & = & g_{j}(x) &\mbox{on} & \partial \Omega,
\end{array}
\right.
$$
and occurs the estimative
\begin{eqnarray}\label{estimativa3.7}
\|u_{j}\|_{W^{2}L^{p,q}_{\omega}(\Omega)}>j(\|f_{j}\|_{L^{p,q}_{\omega}(\Omega)}+\|g_{j}\|_{C^{1,\alpha}(\partial \Omega)}).
\end{eqnarray}
By normalization argument we can assume without loss of generality that $\|u_{k}\|_{W^{2}L^{p,q}_{\omega}(\Omega)}=1$. Consequentely, by \eqref{estimativa3.7} we have that $\|f_{j}\|_{L^{p,q}_{\omega}(\Omega)}\to 0$ and $\|g_{j}\|_{C^{1,\alpha}(\partial \Omega)}\to 0$ when $j\to \infty$. Thus, by estimate \eqref{estimativa3.6}, ABP estimate Lemma \ref{ABP-fullversion}, and Lemma  \ref{Incluaolorentzlebesgue}, 
\begin{eqnarray*}
1=\|u_{j}\|_{W^{2}L^{p,q}_{\omega}(\Omega)}\leq C(\|f_{j}\|_{L^{p,q}_{\omega}(\Omega)}+\|g_{j}\|_{C^{1,\alpha}(\partial \Omega)})\stackrel{j\to \infty}{\longrightarrow} 0,
\end{eqnarray*}
contradition. This finishes the proof, as well as establishes the proof of the Theorem \ref{T1}.
\end{proof}

\section{Some applications of Theorem \ref{T1}}\label{applications}

We conclude the article by making some applications of the Theorem \ref{T1}.We obtain estimates of variable exponent Morrey for the Hessian of solutions of \eqref{1.1}, when $\gamma=g=0$, density of $W^{2}L^{p,q}_{\omega}$ in the fundamental class $S$ of viscosity solutions and global estimates for solutions of the obstacle problem with oblique boundary condition.
\subsection{Variable exponent Morrey regularity}

\hspace{0.4cm} In this part, we will apply the Theorem\ref{T1} in context of Morrey regularity, more precisely, we will guarantee under certain conditions, Morrey Sobolev's estimates for the following problem

\begin{equation}\label{5.1}
\left\{
\begin{array}{rclcl}
F(D^2u,Du,u,x) &=& f(x)& \mbox{in} &   \Omega \\
\mathfrak{B}(Du,0,x)&=& 0 &\mbox{on}& \partial \Omega,
\end{array}
\right.
\end{equation}
in other words, we will work in this section with the problem \eqref{1.1} when $\gamma\equiv g\equiv 0$. For this end, we recall about  \textit{variable exponent Morrey spaces}. Let $\varsigma,\varrho\in C^{0}(\Omega)$ functions such that $0\le \varrho(x)\le \varrho_{0}<n$ and $n<\varsigma_{1}\le \varsigma(x)\le \varsigma_{2}<\infty$ for all $x\in\Omega$, where $\varsigma_{1}$ and $\varsigma_{2}$ are positive constants.The \textit{variable exponent Morrey space} $L^{\varsigma(\cdot),\varrho(\cdot)}(\Omega)$ is the set of mensurable functions $h$ in $\Omega$ such that 
\begin{eqnarray*}
	\rho_{\varsigma(\cdot),\varrho(\cdot)}(h)\defeq\sup_{x\in \Omega,r>0}\left(\frac{1}{r^{\varrho(x)}}\int_{B_{r}(x)}|h(y)|^{\varsigma(y)}dy\right)<\infty
\end{eqnarray*} 
where we will equip with the Luxemburg norm
\begin{eqnarray*}
	\|h\|_{L^{\varsigma(\cdot),\varrho(\cdot)}(\Omega)}\defeq \inf\left\{t>0; \rho_{\varsigma(\cdot),\varrho(\cdot)}\left(\frac{h}{t}\right)\le 1\right\}.
\end{eqnarray*}
Also we define the \textit{variable exponent Morrey Sobolev space} $W^{2}L^{\varsigma(\cdot),\varrho(\cdot)}(\Omega)$ as the set of mensurable functions $h$ in $\Omega$ such that all derivates $D^{\alpha}h$ in distributional sense belongs in $L^{\varsigma(\cdot),\varrho(\cdot)}(\Omega)$ for all  multiindex $\alpha$ with $|\alpha|=0,1,2$ and we will equip this space with norm 
\begin{eqnarray*}
	\|h\|_{W^{2}L^{\varsigma(\cdot),\varrho(\cdot)}(\Omega)}\defeq\sum_{|\alpha|\leq 2}\|D^{\alpha}h\|_{L^{\varsigma(\cdot),\varrho(\cdot)}(\Omega)}
\end{eqnarray*}

 Here, we assume that $\varsigma$ is Log-Hölder continuous, i.e., there exists positive constant $C>0$ such that 
\begin{eqnarray}\label{condicaodeconstantelogholder}
|\varsigma(x)-\varsigma(y)|\leq -C(\log|x-y|)^{-1}, \forall x,y\in \Omega \ \ \mbox{such that} \ \ 0<|x-y|\leq \frac{1}{2}. 
\end{eqnarray}
\begin{remark}
 The Log-Hölder continuity of $\varsigma$ is equivalent to the existence of a modulus of continuity $\iota$ for $\varsigma$ such that  
\begin{eqnarray}\label{condequivalente}
\sup_{0<r\le\frac{1}{2}}\left(\iota(r)\log\left(\frac{1}{r}\right)\right)\leq C_{\varsigma}
\end{eqnarray}
where $C_{\varsigma}$ is a positive constant.\\
Really, assuming that $\varsigma$ is continuous Log-Hölder we know that there exists $\mathrm{C}>0$ such that \eqref{condicaodeconstantelogholder} occurs. Thus, let us define $\iota:[0,+\infty)\to [0,+\infty)$ putting
\begin{eqnarray*}
\iota(t):=\begin{cases}
0,& \ \mbox{if} \ t=0\\
\mathrm{C}\left(\log \frac{1}{t}\right)^{-1},& \ \mbox{if} \ t>0.
\end{cases}
\end{eqnarray*}
Initially, let us note that $\iota$ is well defined, since, due to the continuity of the logarithmic function,
\begin{eqnarray*}
\lim_{t \to 0^{+}}\iota(t)=\lim_{t \to 0^{+}}\mathrm{C}\left(\log\frac{1}{t}\right)^{-1}=0=\iota(0),
\end{eqnarray*}
since $\lim_{t \to 0^{+}}(-\log t)=+\infty$. Furthermore, $\iota$ is increasing and therefore a modulus of continuity and clearly, we have
\begin{eqnarray*}
\sup_{0<t\leq\frac{1}{2}}\left(\iota(t)\log\left(\frac{1}{t}\right)\right)\leq\mathrm{C}
\end{eqnarray*}
Therefore, $\iota$ satisfies \eqref{condequivalente} with $\mathrm{C}_{\varsigma}=\mathrm{C}$. Conversely, assuming that there is $\iota$ is a modulus of continuity for $\varsigma$ satisfying \eqref{condequivalente}, we have for all $x,y\in \Omega$ such that $0<|x-y|\leq\frac{1}{ 2}$, we have
\begin{eqnarray*}
|\varsigma(x)-\varsigma(y)|&\leq& \iota(|x-y|)=\left(\iota(|x-y|)\log\frac{1}{|x-y|}\right)\left(\log\frac{1}{|x-y|}\right)^{-1}\\
&\stackrel{\eqref{condequivalente}}{\leq}&\mathrm{C}_{\varsigma}\frac{1}{\log\frac{1}{|x-y|}}=-\mathrm{C}_{\varsigma}\frac{1}{\log|x-y|},
\end{eqnarray*}
that is, $\varsigma$ is Log-Hölder continuous in $\Omega$, proving the desired equivalence.\\
From this equivalence and the definition of a continuity Log-Hölder, we can conclude that it is equivalent to $\varsigma$ having a modulus of continuity $\tau(t)=\mathrm{C}\left(\log\frac{1} {t}\right)^{-1}$ for some constant $\mathrm{C}>0$.
\end{remark}

From these observations we can state the main application of this subsection:

\begin{theorem}[$W^{2}L^{\varsigma(\cdot),\varrho(\cdot)}$-estimates]\label{morreyestimativa}
Let $\Omega\subset\mathbb{R}^{n}$($n\ge 2$) be an bounded domain with $\partial\Omega\in C^{2,\alpha}$ and assume the structural conditions $(E1),(E3),(E4)$ and $(E5)$. Let $u$ be a $L^{\varsigma_{1}}$-viscosity solution of \eqref{5.1}, where $\beta\in C^{1,\alpha}(\partial \Omega)$ and $f\in L^{\varsigma(\cdot),\varrho(\cdot)}(\Omega)$. Then, $u\in W^{2}L^{\varsigma(\cdot),\varrho(\cdot)}(\Omega)$ with estimate
	\begin{eqnarray}
		\|u\|_{W^{2}L^{\varsigma(\cdot),\varrho(\cdot)}(\Omega)}\le C\cdot \|f\|_{L^{\varsigma(\cdot),\varrho(\cdot)}(\Omega)},
	\end{eqnarray}
	where $C=C(n,\lambda,\Lambda,\mu_{0},\varsigma_{1},\varsigma_{2},\varrho_{0},\alpha,\|\beta\|_{C^{1,\alpha}(\partial\Omega)},\Omega)$.
\end{theorem}
\begin{proof}
The proof follows the same lines as the one in \cite[Theorem 4.2]{ZhangZhenglorentz} with minor modifications. For instance, we must invoke Theorem \ref{T1} instead of \cite[Theorem 1.1]{ZhangZhenglorentz}.
\end{proof}

We have as a consequence of the Theorem \ref{morreyestimativa} a Hölder variable regularity of the gradient for solution viscosity of \eqref{5.1}. First, given be a continuous function  $\alpha:\overline{\Omega}\longrightarrow[0,+\infty)$, we define the \textit{variable exponent Hölder space} $C^{0,\alpha(\cdot)}(\overline{\Omega})$, by set of all functions $u:\overline{\Omega}\longrightarrow\mathbb{R}$ such that,
\begin{eqnarray*}
[u]_{\alpha(\cdot),\overline{\Omega}}:=\sup_{\substack{x,y\in\overline{\Omega}\atop x\neq y}}\frac{|u(x)-u(y)|}{|x-y|^{\alpha(x)}}.
\end{eqnarray*}
with the norm $\|u\|_{C^{0,\alpha(\cdot)}(\overline{\Omega})}=\|u\|_{L^{\infty}(\overline{\Omega})}+[u]_{\alpha(\cdot),\overline{\Omega}}$. 
Now, we have the following result whose proof follows the same proof line as in \cite[Corollary 3.1]{Tang}.
\begin{corollary}
Under assumptions of Theorem $\ref{morreyestimativa}$, let $u\in W^{2,\varsigma(\cdot),\varrho(\cdot)}(\Omega)$ be a $L^{\varsigma_{1}}$-viscosity solution of \eqref{5.1}, with $\varsigma$ and $\varrho$ satisfies the hypothesis above in $\overline{\Omega}$ and $\varsigma(\cdot)+\varrho(\cdot)>n$, then $Du\in C^{0,1-\frac{n-\varrho(\cdot)}{\varsigma(\cdot)}}(\overline{\Omega})$
\end{corollary}

\subsection{Density in the fundamental class $S$}

In this part, we consider the flat problem:
\begin{eqnarray}\label{densidade}
\left\{
\begin{array}{rclcl}
F(D^{2}u,x) & = & f(x) & \text{in} & \mathrm{B}^{+}_{1} \\
\mathfrak{B}(Du,u,x)& = & g(x) & \text{on} & \mathrm{T}_{1}.
\end{array}
\right.
\end{eqnarray}
By Proposition \ref{solution4entradas}, we know that, under structural conditions (E1)-(E5), we have weighted Lorentz-type estimates for solutions of \eqref{densidade}. However, without good asymptotic estimates assumed in conditions (E4)-(E5) we have no guarantee of such regularity for such a problem. In this part, we guarantee in such weaker conditions that we have density of spaces $W^{2}L^{p,q}_{\omega}$ in the fundamental class $S$ for \eqref{densidade} solutions. More precisely, following the proof of \cite[Theorem 6.1]{Bessa} with minimal alterations we guarantee the following theorem:

\begin{theorem}[\bf Density]
Let $u$ be a $C^{0}-$viscosity solution of \eqref{densidade}, where $f\in L^{p,q}_{\omega}(\mathrm{B}^{+}_{1})\cap C^{0}(\mathrm{B}^{+}_{1})$ (for $(p,q)\in(n,\infty)\times(0,\infty]$), $\omega\in \mathfrak{A}_{\frac{p}{n}}$, $\beta,\gamma, g\in C^{1,\alpha}(\mathrm{T}_{1})$ with $\gamma\leq 0$ and $\beta \cdot \textbf{n} \geq \mu_{0}$ on $\mathrm{T}_{1}$, for some $\mu_{0}>0$. Then, for any $\varepsilon>0$, there exists a sequence $(u_{j})_{j\in\mathbb{N}}\subset (W^{2}L^{p,q}_{\omega})_{loc}(\mathrm{B}^{+}_{1})\cap \mathcal{S}(\lambda-\varepsilon,\Lambda+\varepsilon,f)$ converging local uniformly to $u$.
\end{theorem}

\subsection{Estimates for obstacle problem}
Finally, we will use the  Theorem \ref{T1} to study an important free boundary problem, namely the obstacle problem with oblique boundary condition.  More precisely, we will show that under certain conditions solutions of this problem belong to the weighted Lorentz-Sobolev space. The obstacle problem that we will study is the following
\begin{equation} \label{obss1}
\left\{
\begin{array}{rclcl}
F(D^2 u,Du,u,x) &\le& f(x)& \mbox{in} &   \Omega \\
(F(D^2 u, Du, u,x) - f)(u-\psi) &=& 0 &\mbox{in}& \Omega\\
u(x) &\ge& \psi(x) &\mbox{in}& \Omega\\
\mathfrak{B}(Du,u,x)&=& g(x) &\mbox{on}& \partial \Omega,\\
\end{array}
\right.
\end{equation}
where $\psi\in W^{2}L^{p,q}_{\omega}(\Omega)$ is a given obstacle satisfying $\beta(x)\cdot D\psi+\gamma \psi\geq g$ a.e. on $\partial \Omega$. In this line, in \cite{BJ1}, $W^{2,p}$ estimates  were obtained for the problem $\eqref{obss1}$ when $g=0$, $F$ convex and $\psi\in W^{2,p}(\Omega)$. Generalizing this work, \cite{Bessa} showed  $W^{2,p}$ estimates for the same problem \eqref{obss1} under asymptotic conditions for the operator $F$. With this motivation, we will investigate weighted Lorentz-Sobolev regularity estimates for obstacle type
problems \eqref{obss1}.

In what follows, we will assume the following conditions:
\begin{enumerate}
\item[(O1)] There exists a modulus of continuity $\eta: [0,+\infty) \to [0,+\infty)$ with $\eta(0)=0$, such that
$$
F(\mathrm{X}_1, \zeta, r,x_1) - F(\mathrm{X}_2,\zeta,r,x_2) \le \eta\left(|x_1-x_2|\right)\left[(|q| +1) + \alpha_0 |x_1-x_2|^2\right]
$$
for any $x_1,x_2 \in \Omega$, $\zeta \in \mathbb{R}^n$, $r \in \mathbb{R}$, $\alpha_0 >0$ and $\mathrm{X}_1,\mathrm{X}_2 \in \text{Sym(n)}$ satisfying
$$
- 3 \alpha_0
\begin{pmatrix}
	\mathrm{Id}_n& 0 \\
	0& \mathrm{Id}_n
\end{pmatrix}
\leq
\begin{pmatrix}
	\mathrm{X}_2&0\\
	0&-\mathrm{X}_1
\end{pmatrix}
\leq
3 \alpha_0
\begin{pmatrix}
	\mathrm{Id}_n & -\mathrm{Id}_n \\
	-\mathrm{Id}_n& \mathrm{Id}_n
\end{pmatrix},	
$$
where $\mathrm{Id}_n$ is the identity matrix.
\item[(O2)] $F$ is a \textit{proper operator},i.e., $F$  fulfills the following condition,
$$
d\cdot (r_{2}-r_{1}) \leq F(\mathrm{X},\zeta,r_{1},x)-F(\mathrm{X},\zeta,r_{2},x),
$$
for any $\mathrm{X} \in \text{Sym(n)}$, $r_1,r_2 \in \mathbb{R}$, with $r_{1}\leq r_{2}$, $x \in \Omega$, $\zeta \in \mathbb{R}^n$, and some $d >0$.
\end{enumerate}

The conditions $\mbox{(O1)}$ and $\mbox{(O2})$ will be imposed in order to guarantee the \textit{Comparison Principle} for oblique derivative problems and consequently the classic Perron's Method for viscosity solutions(See \cite{Bessa} and \cite{BJ1}).

The main result this section is the following theorem.
\begin{theorem}[{\bf Global estimates for obstacle problems}]\label{T2}
Consider $(p,q)\in(n,\infty)\times (1,\infty)$ and $\tilde{p}\in (n,p)$. Let $u$ be an $L^{\tilde{p}}$-viscosity solution of \eqref{obss1}, where $F$ satisfies the structural assumption (E1)-(E5), $\beta\in C^{2}(\partial \Omega)$ and (O1)-(O2), $\partial \Omega \in C^{3}$ and $\psi \in W^{2}L^{p,q}(\Omega)$. Then, $u \in W^{2}L^{p,q}_{\omega}(\Omega)$  and 
\begin{equation*} \label{obs2}
\|u\|_{W^{2}L^{p,q}_{\omega}(\Omega)} \le \mathrm{C}\cdot \left( \|f\|_{L^{p,q}_{\omega}(\Omega)}+ \|\psi\|_{W^{2}L^{p,q}_{\omega}(\Omega)}  +\|g\|_{C^{1,\alpha}(\partial \Omega)} \right).
\end{equation*}
where $\mathrm{C}=\mathrm{C}(n,\lambda,\Lambda, p,q,\tilde{p},\mu_0, \sigma, \xi, \|\beta\|_{C^2(\partial \Omega)}, \partial \Omega, \textrm{diam}(\Omega), \theta_0)$.
\end{theorem}
\begin{proof}
Consider $\epsilon>0$ arbitrarily fixed and choosen $\Phi_{\epsilon}$ smooth function in $\mathbb{R}$ such that
$$
\Phi_{\varepsilon}(s) \equiv 0 \quad \textrm{if} \quad s \le 0; \quad \Phi_{\varepsilon}(s) \equiv 1 \quad \textrm{if} \quad s \ge \varepsilon,
$$
$$
0 \le \Phi_{\varepsilon}(s) \le 1 \quad \textrm{for any} \quad s \in \mathbb{R}.
$$
and
\begin{eqnarray*}
h(x):=f(x)-F(D^{2}\psi,D\psi,\psi,x).
\end{eqnarray*}
By the $(E1)$ condition and the hypotheses about $\psi$ and $f$, it follows that $h\in L^{p,q}_{\omega}(\Omega)$ and
\begin{eqnarray}\label{5.2}
\|h\|_{L^{p,q}_{\omega}(\Omega)}\leq C\cdot (\|f\|_{L^{p,q}_{\omega}(\Omega)}+\|\psi\|_{W^{2}L^{p,q}_{\omega}(\Omega)})
\end{eqnarray} 
where $C=C(,n,p,q,\lambda,\Lambda,\sigma,\xi)>0$. Now we consider the following penalized problem
\begin{equation} \label{5.3}
\left\{
\begin{array}{rclcl}
F(D^2u_{\epsilon},Du_{\epsilon},u_{\epsilon},x) &=& \mathrm{h}^+(x) \Phi_{\epsilon}(u_{\epsilon} - \psi) + f(x) - \mathrm{h}^+(x)& \mbox{in} & \Omega \\
\mathfrak{B}(Du_{\epsilon},u_{\epsilon},x)&=& g(x)  & \mbox{on} &\partial \Omega,
\end{array}
\right.
\end{equation}
with the aim of finding a unique viscosity solution. For this, for each $v_{0}\in L^{p,q}_{\omega}(\Omega)$ fixed, by Perron's Methods (see Lieberman's book \cite[Theorem 7.19]{Lieberman}) and Theorem \ref{T1} there exists a unique viscosity solution $u_{\epsilon}\in W^{2}L^{p,q}_{\omega}(\Omega)$ for the problem 
\begin{eqnarray*}
\left\{
\begin{array}{rclcl}
F(D^2u_{\epsilon},Du_{\epsilon},u_{\epsilon},x) &=& \mathrm{h}^+(x) \Phi_{\epsilon}(v_{0} - \psi) + f(x) - \mathrm{h}^+(x)& \mbox{in} & \Omega \\
\mathfrak{B}(Du_{\epsilon},u_{\epsilon},x)&=& g(x)  & \mbox{on} &\partial \Omega,
\end{array}
\right.
\end{eqnarray*}
with estimate,
\begin{eqnarray}\label{5.4}
\|u_{\epsilon}\|_{W^{2}L^{p,q}_{\omega}(\Omega)}\leq C\cdot (\|\hat{f}_{v_{0}}\|_{L^{p,q}_{\omega}(\Omega)}+\|g\|_{C^{1,\alpha}(\partial \Omega)}),
\end{eqnarray}
where $C=C(n,p,q,\tilde{p},\lambda,\Lambda\mu_{0},\|\beta\|_{C^{2}(\partial\Omega)},\theta_0,\mbox{diam}(\Omega))>0$  for the function
$$\hat{f}_{v_{0}}(x)=\mathrm{h}^+(x) \Phi_{\epsilon}(v_{0} - \psi) + f(x) - \mathrm{h}^+(x).$$
But, by definitions of $\hat{f}_{v_{0}}$ and $\Phi_{\epsilon}$  and \eqref{5.2},
\begin{eqnarray*}
\{x\in \Omega: |h^{+}(x)\Phi_{\epsilon}(v_{0}(x)-\phi(x))|>t\}\subset \{x\in\Omega: |h^{+}(x)|>t\}, \ \forall t>0
\end{eqnarray*}
consequently,
\begin{eqnarray}\label{estparahmais}
\|h^{+}\Phi_{\epsilon}(v_{0}-\psi)\|_{L^{p,q}_{\omega}(\Omega)}\leq \|h\|_{L^{p,q}_{\omega}(\Omega)}.
\end{eqnarray}
Thus, by estimates \eqref{5.2}  and \eqref{estparahmais}, we obtain
\begin{eqnarray*}
\|\hat{f}_{v_{0}}\|_{L^{p,q}_{\omega}(\Omega)}\leq C(n,p,q,\lambda,\Lambda,\sigma,\xi)\cdot (\|f\|_{L^{p,q}_{\omega}(\Omega)}+\|\psi\|_{W^{2}L^{p,q}_{\omega}(\Omega)}).
\end{eqnarray*}
Thus, we have at \eqref{5.4} that
\begin{eqnarray}\label{5.5}
\|u_{\epsilon}\|_{W^{2}L^{p,q}_{\omega}(\Omega)}\leq C\cdot (\|f\|_{L^{p,q}_{\omega}(\Omega)}+\|\psi\|_{W^{2}L^{p,q}_{\omega}(\Omega)}+\|g\|_{C^{1,\alpha}(\partial \Omega)})\defeq \tilde{C}_{0},
\end{eqnarray}
where $C=C(n,p,q,\tilde{p},\lambda,\Lambda\mu_{0},\|\beta\|_{C^{2}(\partial\Omega)},\theta_0,\mbox{diam}(\Omega))>0$  does not depend on $v_{0}$. This ensures that the  operator $T:L^{p,q}_{\omega}(\Omega)\to W^{2}L^{p,q}_{\omega}(\Omega)\subset L^{p,q}_{\omega}(\Omega)$ given by $T(v_{0})=u_{\epsilon}$ applies closed balls over $\tilde{C}_{0}$-balls. Thus, $T$ is compact operator and consequently we can use the Schauder's fixed point theorem, there exists $u_{\epsilon}$ such that $T(u_{\epsilon})=u_{\epsilon}$, that is, $u_{\epsilon}$ is viscosity solution of $\eqref{5.3}$.

By the construction above we guarantee that the sequence of solutions $\{u_{\epsilon}\}_{\epsilon>0}$ is uniformly bounded on $W^{2}L^{p,q}_{\omega}(\Omega)$. Now by compactness arguments, we can find subsequence $\{u_{\epsilon_{k}}\}_{k\in\mathbb{N}}$ with $\epsilon_{k}\to 0$ and a function $u_{0}\in W^{2}L^{p,q}_{\omega}(\Omega)$ such that $u_{\epsilon_{k}}\rightharpoonup u_{0}$ in $W^{2}L^{p,q}(\Omega)$, $u_{\epsilon_{k}}\to u_{0}$ in $C^{0,\tilde{\alpha}_{0}}$ and $C^{1,\hat{\alpha}}(\Omega)$ for some $\tilde{\alpha}_{0},\hat{\alpha}=\hat{\alpha}(n,\tilde{p})\in (0,1)$. 

Under these conditions we claim that $u_{0}$ is a solution of viscosity of \eqref{obss1}. In fact, by  construction and Arzelá-Ascoli theorem we have that the condition $\mathfrak{B}(Du_{0},u_{0},x)=g(x)$ in viscosity sense is satisfies. On other hand, by Stability Lemma \ref{Est} it follows that, a fact
\begin{eqnarray*}
F(D^2 u_{\epsilon_k}, Du_{\epsilon_k}, u_{\epsilon_k},x) = \mathrm{h}^+(x) \Phi_{\epsilon_k}(u_{\epsilon_k} - \psi) + f(x) - \mathrm{h}^+(x)\le f(x) \quad \textrm{in} \quad \Omega \quad  \forall k \in \mathbb{N}
\end{eqnarray*}
implies that,
\begin{eqnarray*}
F(D^2 u_{0}, Du_{0}, u_{0}, x) \le f \quad \textrm{in} \quad \Omega.
\end{eqnarray*}
Now, we prove that
$$
u_{0} \ge \psi \quad \textrm{in} \quad \overline{\Omega}.
$$
In effect , we note that $\Phi_{\epsilon_k}(u_{\epsilon_k} - \psi) \equiv 0$ on the set $\mathcal{O}_k \defeq \{x \in \overline{\Omega} \, : \, u_{\epsilon_k}(x) < \psi(x)\}$.  Observe that, if $\mathcal{O}_k = \emptyset$, we have nothing to prove. Thus, we suppose that $\mathcal{O}_k \not= \emptyset$. Then,
$$
F(D^2 u_{\epsilon_k}, Du_{\epsilon_k}, u_{\epsilon_{k}},x) = f(x) - \mathrm{h}^+(x) \quad \textrm{for} \,\,\, x \in \mathcal{O}_k.
$$
Now, note that $\mathcal{O}_k$ is relatively open in $\overline{\Omega}$ for each $k \in \mathbb{N}$ since $u_{\epsilon_k} \in C(\overline{\Omega})$. Moreover, from definition of $\mathrm{h}$ we have
$$
F(D^2 \psi, D \psi, \psi, x ) = f(x)-\mathrm{h}(x) \ge F(D^2 u_{\epsilon_k}, Du_{\epsilon_k}, u_{\epsilon_k}, x) \,\,\, \textrm{in} \,\,\, \mathcal{O}_k
$$
Moreover, we have $u_{\epsilon_k} = \psi$ on $\partial \mathcal{O}_k \setminus \partial \Omega$. Thus, from the Comparison Principle (see \cite[Theorem 2.10]{CCKS} and \cite[Theorem 7.17]{Lieberman}) we conclude that $u_{\epsilon_k} \ge \psi$ in $\mathcal{O}_k$,  which is a contradiction. Therefore, $\mathcal{O}_{k} = \emptyset$ and $u_{0} \ge \psi$ in $\overline{\Omega}$.

In order to complete the proof of the claim, we must show that
$$
F(D^2 u_{0}, Du_{0}, u_{0} ,x) = f(x) \,\,\,\, \textrm{in} \,\,\,\, \{ u_{0} > \psi \}
$$
in the viscosity sense. But this fact follows directly from the following observation,
$$
\mathrm{h}^+(x)\Phi_{\epsilon_k}(u_{\epsilon_k} -\psi) + f(x)-\mathrm{h}^{+}(x) \to f(x) \,\,\, \textrm{a.e. in} \,\,\, \left\{x \in \Omega \, : \, u_{0}(x) > \psi(x) + \frac{1}{j} \right\}
$$
and consequently by Stability Lemma \ref{Est} we have that in the viscosity sense
$$
F(D^2 u_{0}, Du_{0}, u_{0} ,x) = f(x) \,\,\, \textrm{in} \,\,\, \{u_{0} > \psi\} = \bigcup_{j=1}^{\infty} \left\{ u_{0} > \psi + \frac{1}{j}\right\} \quad \text{as} \quad  j \to +\infty,
$$
thereby proving the claim.

Thus, from \eqref{5.5} we conclude that $u_{0}$ satisfies
\begin{eqnarray*}
\|u_{0}\|_{W^{2}L^{p,q}_{\omega}(\Omega)} &\le& \liminf_{k \to +\infty} \|u_{\epsilon_k}\|_{W^{2}L^{p,q}_{\omega}(\Omega)} \\ &\le& \mathrm{C}\cdot\left( \|f\|_{L^{p,q}_{\omega}(\Omega)}  + \| \phi \|_{W^{2}L^{p,q}_{\omega}(\Omega)}+ \|g\|_{C^{1,\alpha}(\partial \Omega)}\right)
\end{eqnarray*}
Analogously to \cite[Corollary 5.1]{Bessa} follows the uniqueness of solution of \eqref{obss1} and so consider $u=u_{0}$. This ends the proof of the theorem.
\end{proof}
\begin{remark}
In Theorem $\ref{T2}$, if $p=q$ and $\omega\equiv 1$, then it follows $W^{2,\tilde{p}}$ - estimates for viscosity solutions of the obstacle problem \eqref{obss1} for any $\tilde{p}\in [n,p)$.  
\end{remark}
\subsection*{Declartions}
\paragraph{Acknowledgments} The authors would like to thanks Department of Mathematics of Universidade Federal do Ceará (UFC-
Brazil) by the pleasant scientific and research atmosphere.
\paragraph{Ethical Approval} Not applicable.
\paragraph{Funding}  Gleydson C. Ricarte was partially supported by CNPq-Brazil under Grant No. 304239/2021-6 and 
J.S. Bessa was partially supported by CAPES-Brazil under Grant No. 88887.482068/2020-00.
\paragraph{Availability of data and materials} Data sharing not applicable to this manuscript as no datasets were generated or analysed.


\begin{thebibliography}{99}

\bibitem{BessaOrlicz}
Bessa, J. da Silva. \textit{Weighted Orlicz regularity for fully nonlinear elliptic equations
with oblique derivative at the boundary via asymptotic operators}. Journal of Functional Analysis, v. 286, n. 4, p. 110295, 2024.

\bibitem{Bessa} Bessa, J. da Silva., da Silva, J. V., Frederico, M. N. B. and Ricarte, G. C.
\textit{Sharp Hessian estimates for fully nonlinear elliptic equations under relaxed convexity assumptions, oblique boundary conditions and applications}. Journal of Differential Equations Vol. 367, Pages 451-493.

\bibitem{BJ} Byun, S.S. and Han, J.
\textit{$W^{2,p}$-estimates for fully nonlinear elliptic equations with oblique boundary conditions}.
J. Differential Equations 268 (2020), no. 5, 2125-2150.

\bibitem{BJ1} Byun, S.S., Han J. and  Oh, J.,
\textit{On $W^{2,p}$-estimates for solutions of obstacle problems for fully nonlinear elliptic equations with oblique boundary conditions}.
Calc. Var. 61, 162 (2022).



\bibitem{BJ0} Byun, S.S. and  Wang, L.
\textit{Global Calder\'{o}n-Zygmund theory for asymptotically regular nonlinear elliptic and parabolic equations}.
Int. Math. Res. Not. IMRN 2015, no. 17, 8289-8308.

\bibitem{BOW16} Byun, S.-S.; Oh, J. and Wang, L.
\textit{$W^{2,p}$ estimates for solutions to asymptotically elliptic equations in nondivergence form}.
J. Differential Equations 260 (2016), no. 11, 7965-7981.

\bibitem{Caff1} Caffarelli, L.A.
\textit{Interior a priori estimates for solutions of fully nonlinear equations.}
Ann. of Math.(2) 130 (1989), no. 1, 189--213.

\bibitem{CC} Caffarelli, L.A. and Cabr\'{e}, X.
\textit{Fully Nonlinear Elliptic Equations}. American Mathematical Society Colloquium Publications, 43. American Mathematical Society, Providence, RI, 1995. vi+104 pp. ISBN: 0-8218-0437-5.

\bibitem{CCKS} Caffarelli, L.A., Crandall, M.G.,  Kocan, M. and \'{S}wi\c{e}ch, A.
\textit{On viscosity solutions of fully nonlinear equations with measurable ingredients}.
Comm, Pure Appl. Math. 49 (1996) (4), 365--397.


\bibitem{CM} Cianchi, A. and Maz'ya, V. G.  \textit{Global Lipschitz Regularity for a Class of Quasilinear Elliptic Equations}. Comm. Partial Differential Equations, 36:1, (2010), 100-133.

\bibitem{daSR19} da Silva, J.V. and Ricarte, G.C.
\textit{An asymptotic treatment for non-convex fully nonlinear elliptic equations: Global Sobolev and BMO type estimates.}
Commun. Contemp. Math. 21 (2019), no. 7, 1850053, 28 pp.



\bibitem{DKM} Daskalopoulos, P., Kuusi T. and Mingione, G. \textit{Borderline Estimates for Fully Nonlinear Elliptic Equations}. Comm. in Partial Differential Equations, 39:3, (2014), 574-590.




\bibitem{MF} Foss, M.
\textit{Global regularity for almost minimizers of nonconvex variational problems.}
Ann. Mat. Pura Appl. (4) 187 (2) (2008), 263-321.


\bibitem{Hi} Ishii, H. and  Lions, P.L.
\textit{Fully nonlinear oblique derivative problems for nonlinear second-order elliptic PDEs}.
Duke math. J. 62 (3) (1991) 633-661.



\bibitem{Kry13} Krylov, N.V.
\textit{On the existence of $W^{2,p}$ solutions for fully nonlinear elliptic equations under relaxed convexity assumptions}.
Comm. Partial Diff. Eqs., 38 (2013), pp. 687--710.

\bibitem{Kry17} Krylov, N.V.
\textit{On the existence of $W^{2, p}$ solutions for fully nonlinear elliptic equations under either relaxed or no convexity assumptions}.
Commun. Contemp. Math. 19 (2017), no. 6, 1750009, 39 pp.

\bibitem{LiZhang} Li, D. and Zhang, K.
\textit{Regularity for fully nonlinear elliptic equations with oblique boundary conditions}.
Arch. Ration. Mech. Anal. 228(3)(2018) 923-967.

\bibitem{Lieberman} Lieberman, G.M.
\textit{Oblique derivative problems for elliptic equations}.
World Scientific Publishing Co. Pte. Ltd., Hackensack, NJ, 2013. xvi+509 pp. ISBN: 978-981-4452-32-8.

\bibitem{MaPaVi} Maugeri, Antonino, Palagachev, Dian K. and  Vitanza, Carmela. \textit{A Singular Boundary Value Problem for Uniformly Elliptic Operators}. Journal of Mathematical Analysis and Applications 263.1 (2001), pp. 33-48.

\bibitem{Mengesha} Mengesha, T. and Phuc, N. C.  \textit{Global estimates for quasilinear elliptic equations on Reifenberg flat domains}. Arch. Ration Mech. Anal. 203 (2012) 189–216.

\bibitem{PT} Pimentel, E. and Teixeira, E.V.
\textit{Sharp hessian integrability estimates for nonlinear elliptic equations: an asymptotic approach}.
J. Math. Pures Appl. 106 (2016), pp. 744--767.

\bibitem{JP} Raymond, J.P.
\textit{Lipschitz regularity of solutions of some asymptotically convex problems.}
Proc. Roy. Soc. Edinburgh Sect. A. 117 (1-2) (1991), 59-73.

\bibitem{RT} Ricarte G.C. and  Teixeira, E.V.
\textit{Fully nonlinear singularly perturbed equations and asymptotic free boundary}.
J. Funct. Anal., vol. 261, Issue 6, 2011, 1624-1673.


\bibitem{CS} Scheven, C. and  Schmidt, T.
\textit{Asymptotically regular problems I: Higher integrability.}
J. Differential Equations 248 (4) (2010), 745-791.

\bibitem{CS1} Scheven, C. and  Schmidt, T.
\textit{Asymptotically regular problems II: Partial Lipschitz continuity and a singular set of positive measure.}
Ann. Sc. Norm. Super. Pisa Cl. Sci. (5) 8 (3) (2009), 469-507.


\bibitem{ST} Silvestre, L. and  Teixeira, E.V.
\textit{Regularity estimates for fully non linear elliptic
equations which are asymptotically convex}. Contributions to nonlinear elliptic equations and systems, 425–438, Progr. Nonlinear Differential Equations Appl., 86, Birkh\"{a}user/Springer, Cham, 2015.

\bibitem{TZK}Tian, H., Zheng, S. and Kang, X. \textit{Weighted Lorentz estimates for nonlinear elliptic obstacle problems with partially regular nonlinearities}. Bound. Value. Problems, 2018, 115 (2018). 

\bibitem{Tang} Tang,L.
\textit{$L^{p(x),\lambda(x)}$ regularity for fully nonlinear elliptic equations}. Nonlinear Anal. 149 (2017) 117–129

\bibitem{TN} Tran, M.P. and Nguyen, T.N. \textit{Generalized good-$\lambda$ techniques and applications to weighted Lorentz regularity for quasilinear elliptic equations}. Comptes Rendus. Mathématique, Volume 357 (2019) no. 8, pp. 664-670.
 

\bibitem{Turesson} Turesson, B. O. \textit{Nonlinear potential theory and weighted Sobolev spaces}. Springer-Verlag Berlin Heidelberg New York, 2000.

\bibitem{Winter} Winter, N.
\textit{$W^{2,p}$ and $W^{1,p}$-Estimates at the Boundary for Solutions of Fully Nonlinear, Uniformly Elliptic Equations}.
Z. Anal. Adwend. (J. Anal. Appl.) 28 (2009), 129--164.

\bibitem{ZYY}Zhang J., Yang D. and Yang S. \textit{Weighted Variable Lorentz Regularity for Degenerate Elliptic Equations in Reifenberg Domains}. Math. Meth. Appl. Sci., 2022; 45(14): 8487–8502. 

\bibitem{ZhangZhengmorrey} Zhang, J. and Zheng, S. \textit{Weighted Lorentz and Lorentz-Morrey estimates to viscosity solutions of fully nonlinear elliptic equations}. Complex Variables and Elliptic Equations, (2018),63:9, 1271-1289.

\bibitem{ZhangZhenglorentz} Zhang, J. and Zheng, S. \textit{Weighted Lorentz estimates for fully nonlinear elliptic equations with oblique boundary data}. J. Elliptic Parabolic Equations (2022), 8, 255–281 .

\bibitem{ZZ17} Zhang, J. and Zheng, Z. \textit{Weighted lorentz estimates for nondivergence linear elliptic equations with partially BMO coefficients}. Communications on Pure and Applied Analysis, 2017, 16(3): 899-914.

\bibitem{ZZZ21} Zhang, J., Zheng, S. and  Zuo, C.
\textit{$W^{2,p}$-regularity for asymptotically regular fully nonlinear elliptic and parabolic equations with oblique boundary values}.
Discrete Contin. Dyn. Syst. Ser. S 14 (2021), no. 9, 3305–3318.
\end{thebibliography}
\end{document}